\documentclass[letterpaper,11pt]{amsart}

\usepackage[T1]{fontenc}
\usepackage[utf8]{inputenc}
\usepackage{lmodern}
\usepackage[english]{babel}
\usepackage[autostyle]{csquotes}

\usepackage[pdfusetitle,linktocpage,colorlinks=false]{hyperref}
\usepackage{doi}
\usepackage{filecontents}
\usepackage{enumitem}

\usepackage{amsmath,amssymb,amsthm,amscd}
\usepackage{pictexwd,dcpic}
\newtheorem{thm}{Theorem}[section]
\newtheorem{prop}[thm]{Proposition}
\newtheorem{lem}[thm]{Lemma}

\newtheorem{defn}[thm]{Definition}

\newtheorem{qn}[thm]{Question}
\theoremstyle{remark}
\newtheorem{rem}[thm]{Remark}

\usepackage[colorinlistoftodos]{todonotes}

\newcommand{\supp}{\mathrm{supp}}

\AtEndDocument{\bigskip{\footnotesize%
  \textsc{Institute of Mathematics of the Polish Academy of Sciences, ul. \'{S}niadeckich 8, 00-656 Warszawa, Poland} \par  
  \textit{E-mail address}: \texttt{ychung@impan.pl} \par
}}

\begin{document}
\title[Dynamical Complexity and $L^p$ Operator Crossed Products]{Dynamical Complexity and $K$-Theory of $L^p$ Operator Crossed Products}
\author{Yeong Chyuan Chung}
\date{\today}
\thanks{The author was supported by the European Research Council (ERC-677120).}
\subjclass[2010]{19K56, 46L80, 46L85 47L10, 37B05}
\keywords{Dynamical complexity, $L^p$ operator crossed product, Quantitative operator $K$-theory, Baum-Connes conjecture}
\maketitle

\begin{abstract}
We apply quantitative (or controlled) $K$-theory to prove that a certain $L^p$ assembly map is an isomorphism for $p\in[1,\infty)$ when an action of a countable discrete group $\Gamma$ on a compact Hausdorff space $X$ has finite dynamical complexity. When $p=2$, this is a model for the Baum-Connes assembly map for $\Gamma$ with coefficients in $C(X)$, and was shown to be an isomorphism by Guentner, Willett, and Yu. 
\end{abstract}

\tableofcontents

\section{Introduction}

Notions of dimension abound in mathematics, and they give us quantitative measures of the sizes of various mathematical objects in a broad sense. In some instances, one wishes to know the exact dimension while in other instances, one just wishes to determine finiteness of the dimension. Finiteness of various dimensions has been considered in connection with central problems in the theory of $C^*$-algebras and in noncommutative geometry. For instance, finiteness of nuclear dimension (a noncommutative analog of covering dimension introduced in \cite{WinZac}) plays a crucial role in the classification of $C^*$-algebras (e.g. \cite{Ell,EGLN,EGLN17,Tik,TWW} and references therein), while finiteness of asymptotic dimension (a coarse geometric analog of covering dimension introduced in \cite{Grom}) has featured in work on the Baum-Connes conjecture and the Novikov conjecture (e.g. \cite{CG,CFY,Yu98}). In this paper, we will consider the notion of dynamical complexity and the implication of its finiteness on an $L^p$ assembly map, which is a model for the Baum-Connes assembly map with coefficients when $p=2$.

Dynamical complexity is a property of topological dynamical systems (and more generally, for \'{e}tale groupoids) introduced by Guentner, Willett, and Yu in \cite{GWY2}. Its definition was inspired by the notion of dynamic asymptotic dimension introduced in \cite{GWY1} and by the notion of decomposition complexity introduced in \cite{GTY}. Given an action of a countable discrete group $\Gamma$ on a compact Hausdorff space $X$, the action is said to have finite dynamical complexity if the transformation groupoid $\Gamma\ltimes X$ is contained in the smallest class of open subgroupoids of $\Gamma\ltimes X$ that contains all relatively compact open subgroupoids and is closed under decomposability. Here, decomposability of an open subgroupoid $G$ over a collection $\mathcal{C}$ of open subgroupoids roughly means that at any given scale, there is a cover of the unit space of $G$ by two open sets such that the subgroupoids associated to the two open sets at that scale are both in $\mathcal{C}$. We refer the reader to Definition \ref{FDCbasic} for the precise definition and to \cite[Definition A.4]{GWY2} for a definition applicable to general \'{e}tale groupoids.

The authors of \cite{GWY2} considered a model for the Baum-Connes assembly map for an action based on equivariant versions of Yu's localization algebras (introduced in \cite{Yu97}) and Roe algebras (introduced in \cite{Roe88,Roe93}). In the appendix of \cite{GWY2}, the authors showed that this model for the Baum-Connes assembly map is equivalent to the more traditional one stated in terms of Kasparov's $KK$-theory \cite{BCH}. Their main result is the following:

\begin{thm}\cite{GWY2}
Suppose an action of a countable discrete group $\Gamma$ on a compact Hausdorff space $X$ has finite dynamical complexity. Then $\Gamma$ satisfies the Baum-Connes conjecture with coefficients in $C(X)$.
\end{thm}

Although the aforementioned result follows from earlier work of Tu \cite{Tu} on the Baum-Connes conjecture for amenable groupoids, the proof given in \cite{GWY2} is completely different, and in some sense more direct and more elementary. In fact, their proof is inspired by Yu's proof of the coarse Baum-Connes conjecture for spaces with finite asymptotic dimension in \cite{Yu98}. The main tool in both cases is a controlled Mayer-Vietoris sequence, which is part of a framework of quantitative (or controlled) $K$-theory for $C^*$-algebras  developed by Yu together with Oyono-Oyono in \cite{OY15,OY} (also see \cite{YuQuant} for an overview). Roughly speaking, finite dimension/complexity enables one to apply the Mayer-Vietoris argument a finite number of times to arrive at the quantitative $K$-theory of the algebra in question, and the (standard) $K$-theory of the algebra is obtained as a limit of the quantitative $K$-theory. 

In an earlier paper \cite{Chung1}, we have extended the framework of quantitative $K$-theory to a larger class of Banach algebras, so that it can be applied to algebras of bounded linear operators on $L^p$ spaces in particular. Our goal in this paper is to consider the $L^p$ analog of the assembly map in \cite{GWY2}, and use our extended framework of quantitative $K$-theory to show that this assembly map is an isomorphism under the assumption of finite dynamical complexity. In fact, one sees that the techniques and proofs in \cite{GWY2} carry over to our setting with minor adjustments, the main differences being the exposition of the homotopy invariance argument in Section \ref{sec:Base}, and more care in the use of quantitative $K$-theory (due to an additional norm control parameter) in the Mayer-Vietoris argument used in the proof of the main theorem in Section \ref{sec:Ind}. 

In order to explain our setup, let us briefly recall the Baum-Connes conjecture with coefficients.
Given a (separable) $C^*$-algebra $A$ and an action of a countable discrete group $\Gamma$ on $A$ by $*$-automorphisms, one may form the reduced crossed product $C^*$-algebra $A\rtimes_\lambda\Gamma$. The Baum-Connes conjecture with coefficients \cite{BCH} posits that a certain homomorphism (or assembly map) \[\mu:K_*^\Gamma(\underline{E}\Gamma;A)\rightarrow K_*(A\rtimes_\lambda\Gamma)\] is an isomorphism, where the left-hand side is the equivariant $K$-homology with coefficients in $A$ of the classifying space $\underline{E}\Gamma$ for proper $\Gamma$-actions, and the right-hand side is the $K$-theory of the reduced crossed product $C^*$-algebra. Consider a particular model for $\underline{E}\Gamma$, namely $\bigcup_{s\geq 0}P_s(\Gamma)$ equipped with the $\ell^1$ metric (cf. \cite[Section 2]{BCH}), where $P_s(\Gamma)$ is the Rips complex of $\Gamma$ at scale $s$, i.e., the simplicial complex with vertex set $\Gamma$, and where a finite subset $E\subset\Gamma$ spans a simplex if and only if $d(g,h)\leq s$ for all $g,h\in E$. Here we assume that $\Gamma$ is equipped with a proper length function and $d$ is the associated metric. One may then reformulate the Baum-Connes assembly map as \[\lim_{s\rightarrow\infty}K_*(C_L^*(P_s(\Gamma);A))\stackrel{\epsilon_0}{\rightarrow} \lim_{s\rightarrow\infty}K_*(C^*(P_s(\Gamma);A))\cong K_*(A\rtimes_\lambda\Gamma),\]
where $C^*(P_s(\Gamma);A)$ is the equivariant Roe algebra with coefficients in $A$, $C_L^*(P_s(\Gamma);A)$ is Yu's localization algebra with coefficients in $A$, and $\epsilon_0$ is induced by the evaluation-at-zero map. The fact that $K$-homology can be identified with the $K$-theory of the localization algebra was shown for finite-dimensional simplicial complexes in \cite{Yu97}, and for general locally compact metric spaces in \cite{QR}. The fact that the equivariant Roe algebra with coefficients is stably isomorphic to the reduced crossed product underlies the coarse-geometric approach to the Baum-Connes conjecture with coefficients. We refer the reader to \cite[Appendix B]{GWY2} for a detailed comparison of the abovementioned assembly maps, and also \cite{Roe02} for a comparison of assembly maps without coefficients.

Now let $A$ be a norm-closed subalgebra of $B(L^p(Z,\mu))$ for some measure space $(Z,\mu)$ and $p\in[1,\infty)$. We refer to such algebras as $L^p$ operator algebras. Also suppose that a countable discrete group $\Gamma$ acts on $A$ by isometric automorphisms. By mimicking the $C^\ast$-algebraic definitions, we can define $L^p$ analogs $B^p_L(P_s(\Gamma);A)$, $B^p(P_s(\Gamma);A)$, $A\rtimes_{\lambda,p}\Gamma$ of the localization algebra, the Roe algebra, and the reduced crossed product respectively. We can then consider the following $L^p$ assembly map for the $\Gamma$-action:
\[\lim_{s\rightarrow\infty}K_*(B^p_L(P_s(\Gamma);A))\stackrel{\epsilon_0}{\rightarrow} \lim_{s\rightarrow\infty}K_*(B^p(P_s(\Gamma);A))\cong K_*(A\rtimes_{\lambda,p}\Gamma).\]
In this paper, $A$ will be $C(X)$, where $X$ is a compact Hausdorff space, and
our main result may be stated as follows:
\begin{thm} (cf. Theorem \ref{mainthm}) \label{LpBC}
Suppose that an action of a countable discrete group $\Gamma$ on a compact Hausdorff space $X$ has finite dynamical complexity. Then the $L^p$ assembly map for the action is an isomorphism for $p\in[1,\infty)$.
\end{thm}

A result like this can be seen as an indication of the computability of the $K$-theory of the $L^p$ reduced crossed product on the right-hand side of the assembly map since the $K$-theory of localization-type algebras like the one on the left-hand side of the assembly map has Mayer-Vietoris sequences associated to decompositions of the simplicial complex and other properties of a generalized homology theory (cf. \cite{Yu97,CW}). Moreover, it enables one to transfer questions about whether the $K$-theory of certain $L^p$ reduced crossed products is independent of $p$ (cf. \cite[Problem 11.2]{PhilOpen} and \cite{LY}) over to the left-hand side of the assembly map. In light of interest in $L^p$ operator algebras in recent years (e.g. \cite{ChungLiRigidity,CN,Engel,GardLupGrpoid,GardThielGrp,GardThielConv,Pooya,LY,Phil12,Phil13,PhilBle,PhilViola}), our result is a little step towards our understanding of the $K$-theory of some of these algebras.

On a different note, our method of proof indicates the relative ease with which arguments involving quantitative (or controlled) $K$-theory in the $C^\ast$-algebraic setting may be adapted to the $L^p$ setting, while there is much technical difficulty in adapting other approaches to the Baum-Connes conjecture, such as the Dirac-dual Dirac method, to deal with assembly maps involving $L^p$ operator algebras.


In Section 2, we define $L^p$ Roe algebras and localization algebras associated to an action of a countable discrete group on a compact Hausdorff space, and we define the $L^p$ assembly map in terms of the $K$-theory of these algebras. In the case $p=2$, these are exactly the algebras and map considered in \cite{GWY2}. In Section 3, we associate subalgebras of these algebras to subgroupoids of the transformation groupoid given by the action, and recall the notion of dynamical complexity. In Section 4, we recall some definitions and facts from the framework of quantitative $K$-theory that we developed in \cite{Chung1}. Finally, in Section 5, we prove our main result via a homotopy invariance argument and a Mayer-Vietoris argument.

\section{An $L^p$ assembly map} \label{Sect:assembly}

Throughout this section, $\Gamma$ will be a countable discrete group acting on a compact Hausdorff space $X$ by homeomorphisms. The action will be denoted by $\Gamma\curvearrowright X$. We also assume that $\Gamma$ is equipped with a proper length function $l:\Gamma\rightarrow\mathbb{N}$ and the associated right invariant metric. We will define an assembly map in terms of $L^p$ versions of localization algebras and Roe algebras, where $p\in[1,\infty)$. When $p=2$, we recover (a model for) the Baum-Connes assembly map for $\Gamma$ with coefficients in $C(X)$ considered in \cite{GWY2}.

\begin{defn}
Let $s\geq 0$. The Rips complex of $\Gamma$ at scale $s$, denoted $P_s(\Gamma)$, is the simplicial complex with vertex set $\Gamma$, and where a finite subset $E\subset\Gamma$ spans a simplex if and only if $d(g,h)\leq s$ for all $g,h\in E$.

Points in $P_s(\Gamma)$ can be written as formal linear combinations $\sum_{g\in\Gamma}t_g g$, where $t_g\in[0,1]$ for each $g$ and $\sum_{g\in\Gamma}t_g=1$. We equip $P_s(\Gamma)$ with the $\ell^1$ metric, i.e., $d(\sum_{g\in\Gamma}t_g g,\sum_{g\in\Gamma}s_g g)=\sum_{g\in\Gamma}|t_g-s_g|$.

The barycentric coordinates on $P_s(\Gamma)$ are the continuous functions \[t_g:P_s(\Gamma)\rightarrow[0,1]\] uniquely determined by the condition $z=\sum_{g\in\Gamma}t_g(z)g$ for all $z\in P_s(\Gamma)$.
\end{defn}

Properness of the length function on $\Gamma$ implies that $P_s(\Gamma)$ is finite dimensional and locally compact. Also, the right translation action of $\Gamma$ on itself extends to a right action of $\Gamma$ on $P_s(\Gamma)$ by isometric simplicial automorphisms.

In the usual setting of the Baum-Connes conjecture (e.g. in \cite{GWY2}), one considers Hilbert spaces and $C^*$-algebras encoding the large scale geometry of $\Gamma$ and the topology of $P_s(\Gamma)$. We will replace these Hilbert spaces by $L^p$ spaces, thereby obtaining $L^p$ operator algebras instead of $C^*$-algebras.

First, we recall some facts about $L^p$ tensor products. Details can be found in \cite[Chapter 7]{DF}.

For $p\in[1,\infty)$, there is a tensor product of $L^p$ spaces such that we have a canonical isometric isomorphism $L^p(X,\mu)\otimes L^p(Y,\nu)\cong L^p(X\times Y,\mu\times\nu)$, which identifies, for every $\xi\in L^p(X,\mu)$ and $\eta\in L^p(Y,\nu)$, the element $\xi\otimes\eta$ with the function $(x,y)\mapsto\xi(x)\eta(y)$ on $X\times Y$. Moreover, this tensor product has the following properties:
\begin{enumerate}
\item Under the identification above, the linear span of all $\xi\otimes\eta$ is dense in $L^p(X\times Y,\mu\times\nu)$.
\item $||\xi\otimes\eta||_p=||\xi||_p||\eta||_p$ for all $\xi\in L^p(X,\mu)$ and $\eta\in L^p(Y,\nu)$.
\item The tensor product is commutative and associative.
\item If $a\in B(L^p(X_1,\mu_1),L^p(X_2,\mu_2))$ and $b\in B(L^p(Y_1,\nu_1),L^p(Y_2,\nu_2))$, then there exists a unique \[c\in B(L^p(X_1\times Y_1,\mu_1\times\nu_1),L^p(X_2\times Y_2,\mu_2\times\nu_2))\] such that under the identification above, $c(\xi\otimes\eta)=a(\xi)\otimes b(\eta)$ for all $\xi\in L^p(X_1,\mu_1)$ and $\eta\in L^p(Y_1,\nu_1)$. We will denote this operator by $a\otimes b$. Moreover, $||a\otimes b||=||a|| ||b||$.
\item The tensor product of operators is associative, bilinear, and satisfies $(a_1\otimes b_1)(a_2\otimes b_2)=a_1a_2\otimes b_1b_2$.
\end{enumerate}
If $A\subseteq B(L^p(X,\mu))$ and $B\subseteq B(L^p(Y,\nu))$ are norm-closed subalgebras, we then define $A\otimes B\subseteq B(L^p(X\times Y,\mu\times\nu))$ to be the closed linear span of all elements of the form $a\otimes b$ with $a\in A$ and $b\in B$.

\begin{defn}
For $s\geq 0$, define \[Z_s=\biggl\{\sum_{g\in\Gamma}t_g g\in P_s(\Gamma):t_g\in\mathbb{Q}\;\text{for all}\;g\in\Gamma\biggr\}.\]
Note that $Z_s$ is a $\Gamma$-invariant, countable, dense subset of $P_s(\Gamma)$.

Define \[E_s=\ell^p(Z_s)\otimes\ell^p(X)\otimes\ell^p\otimes\ell^p(\Gamma)\cong \ell^p(Z_s\times X,\ell^p\otimes\ell^p(\Gamma)),\] and equip $E_s$ with the isometric $\Gamma$-action given by \[u_g\cdot(\delta_z\otimes\delta_x\otimes\eta\otimes\delta_h)=\delta_{zg^{-1}}\otimes\delta_{gx}\otimes\eta\otimes\delta_{gh}\] for $z\in Z_s$, $x\in X$, $\eta\in\ell^p$, and $g,h\in\Gamma$. 
\end{defn}

Note that if $s_0\leq s$, then $P_{s_0}(\Gamma)$ identifies equivariantly and isometrically with a subcomplex of $P_s(\Gamma)$, and $Z_{s_0}\subset Z_s$. Hence we have a canonical equivariant isometric inclusion $E_{s_0}\subset E_s$.

We will write $\mathcal{K}_\Gamma$ for the algebra of compact operators on $\ell^p\otimes\ell^p(\Gamma)\cong\ell^p(\mathbb{N}\times\Gamma)$ equipped with the $\Gamma$-action induced by the tensor product of the trivial action on $\ell^p$ and the left regular representation on $\ell^p(\Gamma)$. We also equip the algebra $C(X)\otimes \mathcal{K}_\Gamma$ with the diagonal action of $\Gamma$. Note that the natural faithful representation of $C(X)\otimes \mathcal{K}_\Gamma$ on $\ell^p(X)\otimes\ell^p\otimes\ell^p(\Gamma)$ is covariant for the representation defined by tensoring the natural action on $\ell^p(X)$, the trivial representation on $\ell^p$, and the regular representation on $\ell^p(\Gamma)$. 

Now we can define the $L^p$ operator algebras that will feature in our assembly map.

\begin{defn}
Let $T$ be a bounded linear operator on $E_s$, which we may regard as a ($Z_s\times Z_s$)-indexed matrix $T=(T_{y,z})$ with \[T_{y,z}\in B(\ell^p(X)\otimes\ell^p\otimes\ell^p(\Gamma))\] for each $y,z\in Z_s$.
\begin{enumerate}
\item $T$ is $\Gamma$-invariant if $u_g T u_g^{-1}=T$ for all $g\in\Gamma$, i.e., $T_{y,z}=g\cdot T_{yg,zg}$ for all $g\in\Gamma$.
\item The Rips-propagation of $T$ is $\sup\{d_{P_s(\Gamma)}(y,z):T_{y,z}\neq 0\}$.
\item The $\Gamma$-propagation of $T$, denoted by $\mathrm{prop}_\Gamma(T)$, is \[\sup\{d_\Gamma(g,h):T_{y,z}\neq 0\;\text{for some}\; y,z\in Z_s\;\text{with}\; t_g(y)\neq 0\;\text{and}\;t_h(z)\neq 0\}.\] 
\item $T$ is $X$-locally compact if $T_{y,z}\in C(X)\otimes \mathcal{K}_\Gamma$ for all $y,z\in Z_s$, and if for any compact subset $F\subset P_s(\Gamma)$, the set \[\{(y,z)\in (F\times F)\cap(Z_s\times Z_s):T_{y,z}\neq 0\}\] is finite.
\end{enumerate}
\end{defn}

\begin{defn}
Let $\mathbb{C}[\Gamma\curvearrowright X;s]$ denote the algebra of all $\Gamma$-invariant, $X$-locally compact operators on $E_s$ with finite $\Gamma$-propagation.

Let $B^p(\Gamma\curvearrowright X;s)$ denote the closure of $\mathbb{C}[\Gamma\curvearrowright X;s]$ with respect to the operator norm on $B(E_s)$. We will call $B^p(\Gamma\curvearrowright X;s)$ the (equivariant) $L^p$ Roe algebra of $\Gamma\curvearrowright X$ at scale $s$.
\end{defn}

We will regard the algebras above as concretely represented on $E_s$, and we will often think of elements of $B^p(\Gamma\curvearrowright X;s)$ as matrices $(T_{y,z})_{y,z\in Z_s}$ with entries being continuous equivariant functions $T_{y,z}:X\rightarrow \mathcal{K}_\Gamma$ (with additional properties).

\begin{defn}
Let $\mathbb{C}_L[\Gamma\curvearrowright X;s]$ denote the algebra of all bounded, uniformly continuous functions $a:[0,\infty)\rightarrow\mathbb{C}[\Gamma\curvearrowright X;s]$ such that the $\Gamma$-propagation of $a(t)$ is uniformly finite as $t$ varies, and such that the Rips-propagation of $a(t)$ tends to zero as $t\rightarrow\infty$. 

Let $B^p_L(\Gamma\curvearrowright X;s)$ denote the completion of $\mathbb{C}_L[\Gamma\curvearrowright X;s]$ with respect to the norm \[||a||:=\sup_{t\in[0,\infty)} ||a(t)||_{B^p(\Gamma\curvearrowright X;s)}.\]
We will call $B^p_L(\Gamma\curvearrowright X;s)$ the $L^p$ localization algebra of $\Gamma\curvearrowright X$ at scale $s$.
\end{defn}

We will regard $B^p_L(\Gamma\curvearrowright X;s)$ as concretely represented on $L^p[0,\infty)\otimes E_s$, and elements of $B^p_L(\Gamma\curvearrowright X;s)$ can be regarded as bounded, uniformly continuous functions $a:[0,\infty)\rightarrow B^p(\Gamma\curvearrowright X;s)$ (with additional properties).

Now consider the evaluation-at-zero homomorphism \[\epsilon_0:B^p_L(\Gamma\curvearrowright X;s)\rightarrow B^p(\Gamma\curvearrowright X;s),\] which induces a homomorphism on $K$-theory \[\epsilon_0:K_*(B^p_L(\Gamma\curvearrowright X;s))\rightarrow K_*(B^p(\Gamma\curvearrowright X;s)).\]

If $s_0\leq s$, then the equivariant isometric inclusion $E_{s_0}\subset E_s$ allows us to regard $\mathbb{C}[\Gamma\curvearrowright X;s_0]$ as a subalgebra of $\mathbb{C}[\Gamma\curvearrowright X;s]$. We then regard $B^p(\Gamma\curvearrowright X;s_0)$ (resp. $B^p_L(\Gamma\curvearrowright X;s_0)$) as a subalgebra of $B^p(\Gamma\curvearrowright X;s)$ (resp. $B^p_L(\Gamma\curvearrowright X;s)$). Thus there are directed systems of inclusions of $L^p$ operator algebras $(B^p(\Gamma\curvearrowright X;s))_{s\geq 0}$ and $(B^p_L(\Gamma\curvearrowright X;s))_{s\geq 0}$, and the evaluation-at-zero maps above are compatible with these inclusions.

\begin{defn} \label{assembly}
The $L^p$ assembly map for the action $\Gamma\curvearrowright X$ is the direct limit
\[ \epsilon_0:\lim_{s\rightarrow\infty}K_*(B^p_L(\Gamma\curvearrowright X;s))\rightarrow\lim_{s\rightarrow\infty}K_*(B^p(\Gamma\curvearrowright X;s)). \]
\end{defn}
When $p=2$, this is the model for the Baum-Connes assembly map considered in \cite{GWY2}.

Regarding $A=C(X)$ as an $L^p$ operator algebra (acting by multiplication on $L^p(X,\mu)$ for some measure $\mu$) with an isometric $\Gamma$-action, we can define an $L^p$ reduced crossed product as follows:

Consider $C_c(\Gamma,A)$, the set of finite sums of the form $\sum_{g\in\Gamma}a_g g$ with $a_g\in A$ and with the product given by \[ \biggl(\sum_{g\in\Gamma}a_g g\biggr)\biggl(\sum_{h\in\Gamma}b_h h\biggr)=\sum_{g,h\in\Gamma}a_g\alpha_g(b_h)gh, \] where $\alpha$ denotes the $\Gamma$-action on $A$. There is a natural faithful representation of $C_c(\Gamma,A)$ on $\ell^p(\Gamma,L^p(X,\mu))$ given by
\begin{align*}
(a\xi)(h) &= \alpha_{h^{-1}}(a)\xi(h), \\
(g\xi)(h) &= \xi(g^{-1}h)
\end{align*}
for $a\in A$, $g,h\in\Gamma$, and $\xi\in\ell^p(\Gamma,L^p(X,\mu))$. We then define the $L^p$ reduced crossed product $A\rtimes_{\lambda,p}\Gamma$ to be the operator norm closure of $C_c(\Gamma,A)$ in $B(\ell^p(\Gamma,L^p(X,\mu)))$.

\begin{rem}
For a general $L^p$ operator algebra $A$, it does not seem clear that this way of defining $A\rtimes_{\lambda,p}\Gamma$ is independent of the representation of $A$, which is why Phillips defined the $L^p$ reduced crossed product differently  \cite[Definition 3.3]{Phil13}.
But for a compact metrizable space $X$, one may be able to show that $C(X)\rtimes_{\lambda,p}\Gamma$ is independent of the representation of $C(X)$ on $L^p$ spaces $L^p(Y,\nu)$, at least when the representation is unital and isometric, and the measure $\nu$ is $\sigma$-finite. This is because such representations must factor through the representation of $L^\infty(Y,\nu)$ on $L^p(Y,\nu)$ as multiplication operators \cite[Theorem 4.5]{PhilViola}, a consequence of which is that $C(X)$ has unique $L^p$ operator matrix norms \cite[Proposition 4.6]{PhilViola}.
\end{rem}

One can show that the equivariant $L^p$ Roe algebra $B^p(\Gamma\curvearrowright X;s)$ is isomorphic to $(C(X)\rtimes_{\lambda,p}\Gamma)\otimes K(\ell^p)$. 
Indeed, one obtains an isomorphism between $\mathbb{C}[\Gamma\curvearrowright X;s]$ and $C_c(\Gamma,C(X))\odot K(\ell^p)$ (and between their respective completions) via conjugation by an appropriate invertible operator between $\ell^p$ spaces.
In the special case where there is a bounded fundamental domain $D\subset Z_s$ for the $\Gamma$-action, the invertible operator $U:E_s=\ell^p(Z_s)\otimes\ell^p(X)\otimes\ell^p\otimes\ell^p(\Gamma)\rightarrow \ell^p(\Gamma)\otimes\ell^p(D)\otimes\ell^p(X)\otimes\ell^p\otimes\ell^p(\Gamma)$ given by $\xi\mapsto\sum_{g\in\Gamma}\delta_g\otimes\chi_D u_g\xi$ implements such an isomorphism.

It can also be shown that the canonical homomorphism $C(X)\rtimes_{\lambda,p}\Gamma\rightarrow (C(X)\rtimes_{\lambda,p}\Gamma)\otimes K(\ell^p)$ given by taking the tensor product with a rank one idempotent induces an isomorphism on $K$-theory at least when $p\in(1,\infty)$ by using the same direct limit argument as in the $p=2$ case \cite[Lemma 6.6]{Phil13}.

For most of the rest of this paper, we will work with the kernel of the $L^p$ assembly map, and we introduce the following notation.

\begin{defn}
Let $B^p_{L,0}(\Gamma\curvearrowright X;s)$ be the subalgebra of $B^p_L(\Gamma\curvearrowright X;s)$ consisting of functions $a$ such that $a(0)=0$. We will call $B^p_{L,0}(\Gamma\curvearrowright X;s)$ the $L^p$ obstruction algebra of $\Gamma\curvearrowright X$ at scale $s$.
\end{defn}

\begin{lem}\label{Lem:obsvan}
The $L^p$ assembly map for $\Gamma\curvearrowright X$ is an isomorphism if and only if \[\lim_{s\rightarrow\infty}K_*(B^p_{L,0}(\Gamma\curvearrowright X;s))=0.\]
\end{lem}

\begin{proof}
Just as in the $C^\ast$-algebraic setting, we have a short exact sequence \[0\rightarrow B^p_{L,0}(\Gamma\curvearrowright X;s)\rightarrow B^p_L(\Gamma\curvearrowright X;s)\rightarrow B^p(\Gamma\curvearrowright X;s)\rightarrow 0,\] which induces the usual six-term exact sequence in $K$-theory. The lemma then follows from continuity of $K$-theory under direct limits, and the preservation of exact sequences under direct limits of abelian groups.
\end{proof}

Our goal in this paper will be to show that if $\Gamma\curvearrowright X$ has finite dynamical complexity, then $\lim_{s\rightarrow\infty}K_*(B^p_{L,0}(\Gamma\curvearrowright X;s))=0$, and thus the $L^p$ assembly map for $\Gamma\curvearrowright X$ in Definition \ref{assembly} is an isomorphism.

\section{Groupoids and dynamical complexity}

In this section, we consider the transformation groupoid associated to a group action, its subgroupoids, and $L^p$ operator algebras associated with them. We also recall the definition of dynamical complexity from \cite{GWY2}.

\begin{defn}
The transformation groupoid $\Gamma\ltimes X$ associated to $\Gamma\curvearrowright X$ is $\{(gx,g,x):g\in\Gamma,x\in X\}$ topologized such that the projection $\Gamma\ltimes X\rightarrow\Gamma\times X$ onto the second and third factors is a homeomorphism, and equipped with the following additional structure:
\begin{enumerate}
\item A pair $((hy,h,y),(gx,g,x))$ of elements in $\Gamma\ltimes X$ is said to be composable if $y=gx$. In this case, their product is defined by \[(hgx,h,gx)(gx,g,x)=(hgx,hg,x).\]
\item The inverse of an element $(gx,g,x)\in\Gamma\ltimes X$ is \[(gx,g,x)^{-1}=(x,g^{-1},gx).\]
\item The units of $\Gamma\ltimes X$ are the elements of the clopen subspace \[G^{(0)}=\{(x,e,x):x\in X\},\] where $e$ is the identity in $\Gamma$. We call $G^{(0)}$ the unit space of $\Gamma\ltimes X$.
\end{enumerate}
\end{defn}

\begin{defn}
Let $s\geq 0$, and let $P_s(\Gamma)$ be the Rips complex of $\Gamma$ at scale $s$. The support of $z=\sum_{g\in\Gamma}t_g(z) g\in P_s(\Gamma)$ is the finite set \[\supp(z)=\{g\in\Gamma:t_g(z)\neq 0\}.\]
The support of $T=(T_{y,z})_{y,z\in Z_s}\in B^p(\Gamma\curvearrowright X;s)$ is 
\[ \supp(T)=\left\{(gx,gh^{-1},hx)\in\Gamma\ltimes X:\;\parbox[c][4em][c]{0.5\textwidth}{there exist $y,z\in Z_s$ with $T_{y,z}(x)\neq 0$, $g\in \supp(y)$, and $h\in \supp(z)$}\right\}. \]
\end{defn}

With this definition, one sees that \[\mathrm{prop}_\Gamma(T)=\sup\{l(gh^{-1}):(gx,gh^{-1},hx)\in \supp(T)\;\text{for some}\;x\in X\}.\]

Given two subsets $A,B\subset\Gamma\ltimes X$, we write $AB$ for \[\{ab:a\in A,b\in B,(a,b)\;\text{is composable}\}.\] With this notation, the following lemma says that supports of operators in $B^p(\Gamma\curvearrowright X;s)$ behave as expected under composition of operators.

\begin{lem} \label{supplem}
Let $S,T\in B^p(\Gamma\curvearrowright X;s)$. Then $\supp(ST)\subseteq \supp(S)\supp(T)$.
\end{lem}

\begin{proof}
Suppose that $(gx,gh^{-1},hx)\in \supp(ST)$. Then there are $y,z\in Z_s$ such that $(ST)_{y,z}(x)\neq 0$, $g\in \supp(y)$, and $h\in \supp(z)$. Thus there is $w\in Z_s$ such that $S_{y,w}(x)\neq 0$ and $T_{w,z}(x)\neq 0$. If $k\in \supp(w)$, then $(gx,gk^{-1},kx)\in \supp(S)$ and $(kx,kh^{-1},hx)\in \supp(T)$, so $(gx,gh^{-1},hx)=(gx,gk^{-1},kx)(kx,kh^{-1},hx)\in \supp(S)\supp(T)$.
\end{proof}

\begin{defn}
Let $\Gamma\ltimes X$ be the transformation groupoid associated to $\Gamma\curvearrowright X$. A subgroupoid of $\Gamma\ltimes X$ is a subset $G\subset \Gamma\ltimes X$ that is closed under composition, taking inverses, and units, i.e.,
\begin{enumerate}
\item If $(hgx,h,gx)$ and $(gx,g,x)$ are in $G$, then so is $(hgx,hg,x)$.
\item If $(gx,g,x)\in G$, then $(gx,g,x)^{-1}\in G$.
\item If $(gx,g,x)\in G$, then $(x,e,x)\in G$ and $(gx,e,gx)\in G$, where $e$ is the identity in $\Gamma$.
\end{enumerate}
A subgroupoid is equipped with the subspace topology from $\Gamma\ltimes X$.
\end{defn}

Subgroupoids of $\Gamma\ltimes X$ give rise to subalgebras of the Roe algebra, localization algebra, and obstruction algebra that we defined in the previous section.

\begin{lem}
Let $G$ be an open subgroupoid of $\Gamma\ltimes X$. Define $\mathbb{C}[G;s]$ to be the subspace of $\mathbb{C}[\Gamma\curvearrowright X;s]$ consisting of all operators $T$ with support contained in a compact subset of $G$. Then $\mathbb{C}[G;s]$ is a subalgebra of $\mathbb{C}[\Gamma\curvearrowright X;s]$.
\end{lem}

\begin{proof}
Given Lemma \ref{supplem}, it suffices to show that if $A$ and $B$ are two relatively compact subsets of $G$, then so is $AB$. To see this, first suppose that $A$ and $B$ are compact. Then any net in $AB$ has a convergent subnet since nets in $A$ and nets in $B$ have this property, and so $AB$ is compact. Now if $A$ and $B$ are relatively compact, then since $AB\subset \bar{A}\bar{B}$ and $\bar{A}\bar{B}$ is compact, it follows that $AB$ is relatively compact.
\end{proof}

\begin{defn}
Let $G$ be an open subgroupoid of $\Gamma\ltimes X$. Let $\mathbb{C}_L[G;s]$ denote the subalgebra of $\mathbb{C}_L[\Gamma\curvearrowright X;s]$ consisting of functions $a$ such that $\bigcup_{t\in[0,\infty)} \supp(a(t))$ has compact closure in $G$. 

Let $\mathbb{C}_{L,0}[G;s]$ denote the ideal of $\mathbb{C}_L[G;s]$ consisting of functions $a$ such that $a(0)=0$.

Let $B^p(G;s)$, $B^p_L(G;s)$, and $B^p_{L,0}(G;s)$ denote the respective closures of $\mathbb{C}[G;s]$, $\mathbb{C}_L[G;s]$, and $\mathbb{C}_{L,0}[G;s]$ in $B^p(\Gamma\curvearrowright X;s)$, $B^p_L(\Gamma\curvearrowright X;s)$, and $B^p_{L,0}(\Gamma\curvearrowright X;s)$.
\end{defn}

Note that when $G=\Gamma\ltimes X$, we have $B^p(G;s)=B^p(\Gamma\curvearrowright X;s)$, and similarly for the localization algebra and obstruction algebra.

Since we will be working mostly with the obstruction algebras, we introduce the following shorthand notation for these algebras. We also need to construct filtrations on these algebras in order to apply quantitative $K$-theory in the proof of the main theorem.

\begin{defn}
Let $G$ be an open subgroupoid of $\Gamma\ltimes X$, and let $s\geq 0$. Set $A^s(G)$ to be $B^p_{L,0}(G;s)$. For $r\geq 0$, define \[A^s(G)_r=\{a\in\mathbb{C}_{L,0}[G;s]:\mathrm{prop}_\Gamma(a(t))\leq r\;\text{for all}\;t\},\] which is a linear subspace of $A^s(G)$.

When $G=\Gamma\ltimes X$, we will simply write $A^s$ and $A^s_r$.
\end{defn}

\begin{lem} \label{filtlem}
Let $G$ be an open subgroupoid of $\Gamma\ltimes X$, and let $s\geq 0$. Then the family $(A^s(G))_{r\geq 0}$ of subspaces of $A^s(G)$ satisfies:
\begin{enumerate}
\item if $r_1\leq r_2$, then $A^s(G)_{r_1}\subset A^s(G)_{r_2}$;
\item $A^s(G)_{r_1}A^s(G)_{r_2}\subset A^s(G)_{r_1+r_2}$ for all $r_1,r_2\geq 0$;
\item $\bigcup_{r\geq 0}A^s(G)_r$ is dense in $A^s(G)$.
\end{enumerate}
\end{lem}

\begin{proof}
Note that $a\in A^s(G)_r$ if and only if 
\begin{itemize}
\item $a\in\mathbb{C}_{L,0}[G;s]$, and 
\item $l(g)\leq r$ whenever $(gx,g,x)\in \supp(a(t))$ for some $t\geq 0$.
\end{itemize}
Properties (i) and (iii) follow immediately. 

For (ii), if $a\in A^s(G)_{r_1}$, $b\in A^s(G)_{r_2}$, and $(gx,g,x)\in \supp(a(t)b(t))$ for some $t$, then by Lemma \ref{supplem}, $(gx,g,x)=(gx,gh^{-1},hx)(hx,h,x)$ for some $(gx,gh^{-1},hx)\in \supp(a(t))$ and $(hx,h,x)\in \supp(b(t))$. Thus $l(g)\leq l(gh^{-1})+l(h)\leq r_1+r_2$ so $ab\in A^s(G)_{r_1+r_2}$.
\end{proof}

Note that if $S$ is an open subset of $\Gamma\ltimes X$, then $S$ generates an open subgroupoid of $\Gamma\ltimes X$ (cf. \cite[ Lemma 5.2]{GWY1}).

\begin{defn}
Let $G$ be an open subgroupoid of $\Gamma\ltimes X$, let $H$ be an open subgroupoid of $G$, and let $r\geq 0$. The expansion of $H$ by $r$ relative to $G$, denoted by $H^{+r}$, is the open subgroupoid of $\Gamma\ltimes X$ generated by \[H\cup\{(gx,g,x)\in G:x\in H^{(0)},l(g)\leq r\}.\]
\end{defn}

Note that $H^{+r}$ depends on $G$ although we do not indicate this in the notation.

\begin{lem} \label{pdtlem}
Let $G$ be an open subgroupoid of $\Gamma\ltimes X$, let $H$ be an open subgroupoid of $G$, and let $r,s\geq 0$. Then \[A^s(H)\cdot A^s_r(G)\cup A^s_r(G)\cdot A^s(H)\subseteq A^s(H^{+r}).\]
\end{lem}

\begin{proof}
Follows from Lemma \ref{supplem}.
\end{proof}

\begin{lem}
Let $G$ be an open subgroupoid of $\Gamma\ltimes X$, let $H$ be an open subgroupoid of $G$, and let $r_1,r_2\geq 0$. Then $(H^{+r_1})^{+r_2}\subseteq H^{+(r_1+r_2)}$.
\end{lem}

\begin{proof}
It suffices to show that \[\{(gx,g,x)\in G:x\in (H^{+r_1})^{(0)},l(g)\leq r_2\}\subseteq H^{+(r_1+r_2)}.\] Pick such an element $(gx,g,x)$. There exists $h\in\Gamma$ with $l(h)\leq r_1$ and $hx\in H^{(0)}$. Thus $(gx,gh^{-1},hx)$ and $(hx,h,x)$ are in $H^{+(r_1+r_2)}$ so $(gx,g,x)=(gx,gh^{-1},hx)(hx,h,x)\in H^{+(r_1+r_2)}$.
\end{proof}

Now we recall the definition of dynamical complexity for transformation groupoids. We refer the reader to \cite[Appendix A]{GWY2} for a definition applicable to general \'{e}tale groupoids.

\begin{defn}\cite{GWY2} \label{FDCbasic}
Let $\Gamma\curvearrowright X$ be an action, let $G$ be an open subgroupoid of $\Gamma\ltimes X$, and let $\mathcal{C}$ be a set of open subgroupoids of $\Gamma\ltimes X$. 
We say that $G$ is decomposable over $\mathcal{C}$ if for all $r\geq 0$ there exists an open cover $G^{(0)}=U_0\cup U_1$ of the unit space of $G$ such that for each $i\in\{0,1\}$ the subgroupoid of $G$ generated by
\[ \{ (gx,g,x)\in G:x\in U_i,l(g)\leq r\} \]
is in $\mathcal{C}$.

An open subgroupoid of $\Gamma\ltimes X$ is said to have finite dynamical complexity if it is contained in the smallest class $\mathcal{D}$ of open subgroupoids of $\Gamma\ltimes X$ that contains all relatively compact open subgroupoids and is closed under decomposability (i.e., if $G$ decomposes over $\mathcal{D}$, then $G$ is in $\mathcal{D}$).

The action $\Gamma\curvearrowright X$ is said to have finite dynamical complexity if $\Gamma\ltimes X$ has finite dynamical complexity.
\end{defn}

The following is a slight variation of the definition.

\begin{defn}\cite{GWY2}
Let $\Gamma\curvearrowright X$ be an action, let $G$ be an open subgroupoid of $\Gamma\ltimes X$, and let $\mathcal{C}$ be a set of open subgroupoids of $\Gamma\ltimes X$. 
We say that $G$ is strongly decomposable over $\mathcal{C}$ if for all $r\geq 0$ there exists an open cover $G^{(0)}=U_0\cup U_1$ of the unit space of $G$ such that for each $i\in\{0,1\}$, if $G_i$ is the subgroupoid of $G$ generated by
\[ \{ (gx,g,x)\in G:x\in U_i,l(g)\leq r\}, \]
then $G_i^{+r}$ (with expansion taken relative to $G$) is in $\mathcal{C}$.

Denote by $\mathcal{D}_s$ the smallest class of open subgroupoids of $\Gamma\ltimes X$ that contains all relatively compact open subgroupoids and is closed under strong decomposability (i.e., if $G$ is strongly decomposable over $\mathcal{D}_s$, then $G$ is in $\mathcal{D}_s$).
\end{defn}

The following lemma describes properties of finite dynamical complexity that we will use later. We refer the reader to \cite[Lemma 3.16]{GWY2} for the proof.

\begin{lem} \label{FDCstrong} 
Let $\Gamma\curvearrowright X$ be an action.
\begin{enumerate}
\item If $G$ is an open subgroupoid of $\Gamma\ltimes X$ that is in $\mathcal{D}$ (resp. $\mathcal{D}_s$), then any open subgroupoid of $G$ is also in $\mathcal{D}$ (resp. $\mathcal{D}_s$).
\item $\mathcal{D}=\mathcal{D}_s$.
\end{enumerate}
\end{lem}

\section{Quantitative $K$-theory}

In this section, we recall some definitions and facts from our framework of quantitative (or controlled) $K$-theory in \cite{Chung1}.
\begin{defn}
A filtered Banach algebra is a Banach algebra $A$ with a family $(A_r)_{r\geq 0}$ of linear subspaces such that
\begin{enumerate}
\item $A_{r_1}\subset A_{r_2}$ if $r_1\leq r_2$;
\item $A_{r_1}A_{r_2}\subset A_{r_1+r_2}$ for all $r_1,r_2\geq 0$;
\item $\bigcup_{r\geq 0}A_r$ is dense in $A$.
\end{enumerate}
If $A$ is unital with unit $1_A$, then we require $1_A\in A_r$ for all $r\geq 0$.
\end{defn}

We showed in Lemma \ref{filtlem} that if $G$ is an open subgroupoid of $\Gamma\ltimes X$ and $s\geq 0$, then $A^s(G)$ is a filtered Banach algebra with filtration \[A^s(G)_r=\{a\in\mathbb{C}_{L,0}[G;s]:\mathrm{prop}_\Gamma(a(t))\leq r\;\text{for all}\;t\}.\]

We will write $\tilde{A}$ for the algebra obtained from $A$ by adjoining a unit if $A$ is non-unital, and $A$ itself if $A$ is already unital.

\begin{defn}
Let $A$ be a filtered Banach algebra. For $0<\varepsilon<\frac{1}{20}$, $r\geq 0$, and $N\geq 1$,
\begin{enumerate}
\item an element $e\in A$ is called an $(\varepsilon,r,N)$-idempotent if $||e^2-e||<\varepsilon$, $e\in A_r$, and $\max(||e||,||1_{\tilde{A}}-e||)\leq N$.
\item if $A$ is unital, an element $u\in A$ is called an $(\varepsilon,r,N)$-invertible if $u\in A_r$, $||u||\leq N$, and there exists $v\in A_r$ with $||v||\leq N$ such that $\max(||uv-1||,||vu-1||)<\varepsilon$.
\end{enumerate}
\end{defn}

\begin{defn}
Let $A$ be a filtered Banach algebra.
\begin{enumerate}
\item Two $(\varepsilon,r,N)$-idempotents $e_0$ and $e_1$ in $A$ are $(\varepsilon',r',N')$-homotopic for some $\varepsilon'\geq\varepsilon$, $r'\geq r$, and $N'\geq N$ if there exists a norm-continuous path $(e_t)_{t\in[0,1]}$ of $(\varepsilon',r',N')$-idempotents in $A$ from $e_0$ to $e_1$. 
Equivalently, there is an $(\varepsilon',r',N')$-idempotent $e\in C([0,1],A)$ such that $e(0)=e_0$ and $e(1)=e_1$.
\item If $A$ is unital, two $(\varepsilon,r,N)$-invertibles $u_0$ and $u_1$ in $A$ are $(\varepsilon',r',N')$-homotopic for some $\varepsilon'\geq\varepsilon$, $r'\geq r$, and $N'\geq N$ if there exists an $(\varepsilon',r',N')$-invertible $u$ in $C([0,1],A)$ with $u(0)=u_0$ and $u(1)=u_1$. In this case, we get a norm-continuous path $(u_t)_{t\in[0,1]}$ of $(\varepsilon',r',N')$-invertibles in $A$ from $u_0$ to $u_1$ by setting $u_t=u(t)$.
\end{enumerate}
\end{defn}

Using these notions of homotopy, we may consider the following equivalence relations on the $(\varepsilon,r,N)$-idempotents and $(\varepsilon,r,N)$-invertibles in $M_\infty(A)$.

\begin{defn}
Let $A$ be a filtered Banach algebra.
\begin{enumerate}
\item Two $(\varepsilon,r,N)$-idempotents in $M_\infty(A)$ are said to be equivalent if they are $(4\varepsilon,r,4N)$-homotopic.
\item If $A$ is unital, two $(\varepsilon,r,N)$-invertibles in $M_\infty(A)$ are said to be equivalent if they are $(4\varepsilon,2r,4N)$-homotopic.
\end{enumerate}
\end{defn}

For a unital Banach algebra $A$, the operation $[e]+[f]:=\left[\begin{pmatrix} e & 0 \\ 0 & f \end{pmatrix}\right]$ makes the set of equivalence classes of $(\varepsilon,r,N)$-idempotents in $M_\infty(A)$ an abelian semigroup. The Grothendieck group of this abelian semigroup is defined to be $K_0^{\varepsilon,r,N}(A)$.
Also, the operation $[u]+[v]:=\left[\begin{pmatrix} u & 0 \\ 0 & v \end{pmatrix}\right]$ makes the set of equivalence classes of $(\varepsilon,r,N)$-invertibles in $M_\infty(A)$ an abelian group, and this is defined to be $K_1^{\varepsilon,r,N}(A)$.

For a non-unital Banach algebra $A$, consider the unitization $\tilde{A}$ of $A$, and define $K_*^{\varepsilon,r,N}(A)$ to be the kernel of the map $K_*^{\varepsilon,r,N}(\tilde{A})\rightarrow K_*^{\varepsilon,r,N}(\mathbb{C})$ induced by the canonical homomorphism $\tilde{A}\rightarrow\mathbb{C}$, where $\mathbb{C}$ is given the trivial filtration.

We refer the reader to \cite[Section 3]{Chung1} for details.

There are canonical homomorphisms \[\iota_*^{\varepsilon,\varepsilon',r,r',N,N'}:K_*^{\varepsilon,r,N}(A)\rightarrow K_*^{\varepsilon',r',N'}(A)\] for $0<\varepsilon\leq\varepsilon'<\frac{1}{20}$, $0\leq r\leq r'$, and $1\leq N\leq N'$, which we may think of as relaxation of control maps or inclusion maps.

If $e$ is an $(\varepsilon,r,N)$-idempotent in a unital filtered Banach algebra $A$, then we may apply the holomorphic functional calculus to get an idempotent $c_0(e)\in A$. This gives us a group homomorphism \[c_0:K_0^{\varepsilon,r,N}(A)\rightarrow K_0(A).\] Also, every $(\varepsilon,r,N)$-invertible is actually invertible so we have a group homomorphism \[c_1:K_1^{\varepsilon,r,N}(A)\rightarrow K_1(A)\] given by $[u]_{\varepsilon,r,N}\mapsto[u]$. We sometimes refer to these homomorphisms as comparison maps.
These homomorphisms allow us to pass from quantitative $K$-theory to standard $K$-theory. 
To pass from standard $K$-theory to quantitative $K$-theory, we may use the following two propositions.

\begin{prop}\cite[Proposition 3.20]{Chung1} \leavevmode \label{qKtoKsurj}
\begin{enumerate}
\item Let $A$ be a filtered $L^p$ operator algebra. Let $f$ be an idempotent in $M_n(\tilde{A})$, and let  $0<\varepsilon<\frac{1}{20}$. Then there exist $r\geq 0$ and $[e]\in K_0^{\varepsilon,r,||f||+1}(A)$ with $e$ an $(\varepsilon,r,||f||+1)$-idempotent in $M_n(\tilde{A})$ such that $c_0([e])=[f]$ in $K_0(A)$. 
\item Let $A$ be a filtered $L^p$ operator algebra. Let $u$ be an invertible element in $M_n(\tilde{A})$, and let $0<\varepsilon<\frac{1}{20}$. Then there exist $r\geq 0$ and $[v]\in K_1^{\varepsilon,r,||u||+||u^{-1}||+1}(A)$ with $v$ an $(\varepsilon,r,||u||+||u^{-1}||+1)$-invertible in $M_n(\tilde{A})$ such that $c_1([v])=[u]$ in $K_1(A)$.
\end{enumerate}
\end{prop}

\begin{prop}\cite[Proposition 3.21]{Chung1} \leavevmode \label{qKtoKinj}
\begin{enumerate}
\item There exists a non-decreasing function $P:[1,\infty)\rightarrow[1,\infty)$ such that for any filtered $L^p$ operator algebra $A$, if $0<\varepsilon<\frac{1}{20P(N)}$, and $[e]_{\varepsilon,r,N},[f]_{\varepsilon,r,N}\in K_0^{\varepsilon,r,N}(A)$ are such that $c_0([e])=c_0([f])$ in $K_0(A)$, then there exist $r'\geq r$ and $N'\geq N$ such that in $K_0^{P(N)\varepsilon,r',N'}(A)$ we have $[e]_{P(N)\varepsilon,r',N'}=[f]_{P(N)\varepsilon,r',N'}$.
\item Let $A$ be a filtered $L^p$ operator algebra. Suppose that $0<\varepsilon<\frac{1}{20}$, and $[u]_{\varepsilon,r,N},[v]_{\varepsilon,r,N}\in K_1^{\varepsilon,r,N}(A)$ are such that $c_1([u])=c_1([v])$ in $K_1(A)$. Then there exist $r'\geq r$ and $N'\geq N$ such that in $K_1^{\varepsilon,r',N'}(A)$ we have $[u]_{\varepsilon,r',N'}=[v]_{\varepsilon,r',N'}$.
\end{enumerate}
\end{prop}

\begin{defn}
A control pair is a pair $(\lambda,h)$, where
\begin{enumerate}
\item $\lambda:[1,\infty)\rightarrow[1,\infty)$ is a non-decreasing function;
\item $h:(0,\frac{1}{20})\times[1,\infty)\rightarrow[1,\infty)$ is a function such that $h(\cdot,N)$ is non-increasing for fixed $N$. \end{enumerate}
We will write $\lambda_N$ for $\lambda(N)$, and $h_{\varepsilon,N}$ for $h(\varepsilon,N)$.

Given two control pairs $(\lambda,h)$ and $(\lambda',h')$, we write $(\lambda,h)\leq(\lambda',h')$ if $\lambda_N\leq \lambda'_N$ and $h_{\varepsilon,N}\leq h'_{\varepsilon,N}$ for all $\varepsilon\in(0,\frac{1}{20})$ and $N\geq 1$.
\end{defn}

In \cite[Section 5]{Chung1}, we showed the existence of a controlled Mayer-Vietoris sequence for a pair of subalgebras satisfying certain conditions that, roughly speaking, control the decomposition of elements in the larger algebra into sums of elements in the subalgebras and also controls approximation of elements in the subalgebras by elements in the intersection. Here, in order to slightly simplify the statement, we use less general hypotheses (by considering pairs of ideals and omitting certain parameters) that suffice for our application. On the other hand, we also give ourselves a bit more flexibility in terms of propagation control. One can check that the proofs in \cite{Chung1} carry over with appropriate adjustments of the propagation parameter $r$. In \cite{GWY2} there is a slightly different approach for obtaining a controlled Mayer-Vietoris sequence in the $C^\ast$-algebra setting. 

\begin{defn} \label{MVpair}
Let $A$ be a filtered $L^p$ operator algebra with filtration $(A_r)_{r\geq 0}$. A pair $(I,J)$ of closed ideals of $A$ is a controlled Mayer-Vietoris pair for $A$ if it satisfies the following conditions:
\begin{enumerate}
\item There exists $\rho:[0,\infty)\rightarrow[0,\infty)$ with $\rho(r)\geq r$ such that for any $r\geq 0$, any positive integer $n$, and any $x\in M_n(A_r)$, there exist $x_1\in M_n(I\cap A_{\rho(r)})$ and $x_2\in M_n(J\cap A_{\rho(r)})$ such that $x=x_1+x_2$ and $\max(||x_1||,||x_2||)\leq ||x||$;
\item $I$ and $J$ have filtrations $(I\cap A_r)_{r\geq 0}$ and $(J\cap A_r)_{r\geq 0}$ respectively;
\item For any $r\geq 0$, any $\varepsilon>0$, any positive integer $n$, any $x\in M_n(I\cap A_r)$ and $y\in M_n(J\cap A_r)$ with $||x-y||<\varepsilon$, there exists $z\in M_n(I\cap J\cap A_{\rho(r)})$ such that $\max(||z-x||,||z-y||)<\varepsilon$, where $\rho$ is as above.
\end{enumerate}
\end{defn}

\begin{thm}\cite[Definition 5.12 and Theorem 5.14]{Chung1} \label{MVthm}
There exist control pairs $(\lambda,h)\leq(\lambda',h')$ such that for any filtered $L^p$ operator algebra $A$ and any controlled Mayer-Vietoris pair $(I,J)$ for $A$, if $x\in K_1^{\varepsilon,r,N}(A)$, then there exists \[\partial_c(x)\in K_0^{\lambda_N\varepsilon,h_{\varepsilon,N}r,\lambda_N}(I\cap J)\] such that if $\partial_c(x)=0$, then there exist $a\in K_0^{\lambda'_N\varepsilon,h'_{\varepsilon,N}r,\lambda'_N}(I)$ and $b\in K_0^{\lambda'_N\varepsilon,h'_{\varepsilon,N}r,\lambda'_N}(J)$ such that $x=a+b$ in $K_0^{\lambda'_N\varepsilon,h'_{\varepsilon,N}r,\lambda'_N}(A)$.
\end{thm}

\begin{rem}
Note that the theorem above only mentions a boundary map from the quantitative $K_1$ group of $A$ to the quantitative $K_0$ group of $I\cap J$, and not the other bounday map from the quantitative $K_0$ group of $A$ to the quantitative $K_1$ group of $I\cap J$. This is because a controlled version of Bott periodicity was not proved in \cite{Chung1}.
Nevertheless, in our application, we may consider suspensions of the algebras involved and still apply the theorem.
\end{rem}

\section{Main theorem}

In this section, we shall prove the main theorem that the $L^p$ assembly map for $\Gamma\curvearrowright X$ in Definition \ref{assembly} is an isomorphism if $\Gamma\curvearrowright X$ has finite dynamical complexity.
The proof that we present is modeled after the proof in the $C^*$-algebraic setting in \cite{GWY2}, consisting of a homotopy invariance argument and a Mayer-Vietoris argument.


\subsection{Homotopy invariance} \label{sec:Base}

Recall that we use the shorthand $A^s(G)$ for $B^p_{L,0}(G;s)$. The homotopy invariance argument involved in the proof of the main theorem leads to the following statement. 

\begin{prop} \label{baseprop}
Let $G$ be an open subgroupoid of $\Gamma\ltimes X$ such that \[G\subseteq\{(gx,g,x)\in\Gamma\ltimes X:l(g)\leq s\}\] for some $s\geq 0$. Then $K_*(A^s(G))=0$.
\end{prop}

Before getting into the proof of the proposition, we need to fix some terminology that is standard in the $C^*$-algebraic setting but perhaps less so when Hilbert spaces are replaced by other Banach spaces. Having done that, the series of results in the rest of this section then yields the proposition.


\begin{defn}
Let $E$ be a complex Banach space. We say that $T\in B(E)$ is a partial isometry if $||T||\leq 1$ and there exists $S\in B(E)$ such that $||S||\leq 1$, $TST=T$, and $STS=S$. We call such an $S$ a generalized inverse of $T$.
\end{defn}

\begin{rem}	\leavevmode
\begin{enumerate}
\item In \cite[Section 6]{Phil12}, Phillips considers spatial partial isometries on $L^p$ spaces. Such spatial partial isometries are partial isometries in the sense of the preceding definition but the converse is not true.
\item If $(Z,\mu)$ is a $\sigma$-finite measure space, $p\in[1,\infty)\setminus\{2\}$, and $T\in B(L^p(Z,\mu))$ is an isometric (but not necessarily surjective) linear map, then it follows from Lamperti's theorem \cite{Lamp} (also see \cite[Theorem 6.9]{Phil12}) that $T$ is a partial isometry in the sense above, and one can find a generalized inverse $S$ such that $ST=I$ (cf. \cite[Section 6]{Phil12}). Hereafter, we will denote a fixed choice of such an $S$ by $T^\dagger$.
\end{enumerate}
\end{rem}

\begin{defn}
If $A\subset B(L^p(\mu))$ is an $L^p$ operator algebra, then we say that $b\in B(L^p(\mu))$ is a multiplier of $A$ if $bA\subset A$ and $Ab\subset A$. We say that $b$ is an isometric multiplier of $A$ if $b$ is an isometry and both $b$ and $b^\dagger $ are multipliers of $A$. Denote by $M(A)$ the set of all multipliers of $A$.
\end{defn}

Note that $M(A)$ is also an $L^p$ operator algebra.

For $G$ an open subgroupoid of $\Gamma\ltimes X$ and $s\geq 0$, we introduce the following notation: 
\begin{itemize}
\item $P_s(G)=\{(z,x)\in P_s(\Gamma)\times X:(gx,g,x)\in G\;\text{for all}\;g\in \supp(z)\}$.
\item $Z_G=(Z_s\times X)\cap P_s(G)$.
\item $E_G=\ell^p(Z_G,\ell^p\otimes\ell^p(\Gamma))\cong\ell^p(Z_G)\otimes\ell^p\otimes\ell^p(\Gamma)$.
\end{itemize}

Note that $E_G$ is a subspace of $E_s$. 
Moreover, the faithful representation of $B^p(G;s)$ on $E_s$ restricts to a faithful representation on $E_G$. 
Thus we will regard $B^p(G;s)$ as faithfully represented on $E_G$, and $A^s(G):=B^p_{L,0}(G;s)$ as faithfully represented on $L^p([0,\infty),E_G)$.

If $(z,x)\in P_s(G)$ and $\supp(z)=\{g_1,\ldots,g_n\}$, then $\{e,g_1,\ldots,g_n\}$ also spans a simplex $\Delta$ in $P_s(\Gamma)$ such that $\Delta\times\{x\}$ is contained in $P_s(G)$. Hence the family of functions \begin{equation}\label{homodef} F_r:P_s(G)\rightarrow P_s(G), (z,x)\mapsto ((1-r)z+re,x) \quad\quad (0\leq r\leq 1) \end{equation} defines a homotopy between the identity map on $P_s(G)$ and the projection onto the subset $\{(z,x)\in P_s(G):z=e\}$, which we may identify with the unit space $G^{(0)}$.

In the definition of $\mathbb{C}[\Gamma\curvearrowright X;s]$, we may use $\mathcal{K}_\Gamma^\infty$, the algebra of compact operators on $(\bigoplus_{n=0}^\infty \ell^p\otimes\ell^p(\Gamma))_{p}\cong (\bigoplus_{n=0}^\infty\ell^p)_{p}\otimes\ell^p(\Gamma)$, thereby obtaining another $L^p$ Roe algebra $B^p(\Gamma\curvearrowright X;\mathcal{K}_\Gamma^\infty;s)$. Moreover, fixing an isometric isomorphism $\phi:\ell^p\stackrel{\cong}{\rightarrow}(\bigoplus_{n=0}^\infty\ell^p)_{p}$ gives an isomorphism \[B^p(\Gamma\curvearrowright X;s)\cong B^p(\Gamma\curvearrowright X;\mathcal{K}_\Gamma^\infty;s).\]
We also have the corresponding statements for the $L^p$ localization algebras and obstruction algebras defined earlier. 

For each $n$, define an isometry $u_{n,0}:\ell^p\rightarrow(\bigoplus_{n=0}^\infty\ell^p)_{p}$ by inclusion as the $n$th summand, and define $u_{n,0}^\dagger :(\bigoplus_{n=0}^\infty\ell^p)_{p}\rightarrow\ell^p$ by projection onto the $n$th summand. Then $u_{n,0}^\dagger u_{n,0}=I$ for all $n$, and $u_{n,0}^\dagger u_{m,0}=0$ when $n\neq m$. Define $u_n:L^p([0,\infty),E_G)\rightarrow L^p([0,\infty),E_G^\infty)$ to be the operator induced by tensoring $u_{n,0}$ with the identity on the other factors, where $E_G^\infty=(\bigoplus_{n=0}^\infty E_G)_{p}$, and define $u_n^\dagger $ similarly using $u_{n,0}^\dagger $. Then $u_n^\dagger u_n=I$ for all $n$, and $u_n^\dagger u_m=0$ when $n\neq m$. 

Given $a\in B(\ell^p)$, consider $a^\infty=a\oplus a\oplus\cdots\in B((\bigoplus_{n=0}^\infty\ell^p)_{p})$. Then $\mu(a)=a^\infty$ is an isometric homomorphism $B(\ell^p)\rightarrow B((\bigoplus_{n=0}^\infty\ell^p)_{p})$. We may also consider the isometry $v\in B((\bigoplus_{n=0}^\infty\ell^p)_{p})$ given by the right shift taking the $n$th summand onto the $(n+1)$st summand. Denote by $v^\dagger \in B((\bigoplus_{n=0}^\infty\ell^p)_{p})$ the left shift. Then $v\mu(a)v^\dagger =0\oplus a\oplus a\oplus\cdots$ for all $a\in B(\ell^p)$. With $u_{0,0}$ as above, $u_{0,0}au_{0,0}^\dagger $ is given by $a\oplus 0\oplus 0\oplus\cdots$ for all $a\in B(\ell^p)$. Now $a\mapsto\mu^{+1}(a):=v\mu(a)v^\dagger $ and $a\mapsto\mu^0(a):=u_{0,0}au_{0,0}^\dagger $ are bounded homomorphisms $B(\ell^p)\rightarrow B((\bigoplus_{n=0}^\infty\ell^p)_{p})$.

Now given $a\in B(L^p([0,\infty),E_G))$, consider the bounded linear operator $\mu(a)=a^\infty$ on $L^p([0,\infty),E_G^\infty)$. Proceeding similarly as above, we get bounded homomorphisms \[\mu,\mu^{+1},\mu^0:B(L^p([0,\infty),E_G))\rightarrow B(L^p([0,\infty),E_G^\infty)).\] Moreover, each of them maps 
$M(A^s(G))$ into $M(A^s(G;\mathcal{K}_\Gamma^\infty))$.

The following lemma is an $L^p$ version of a fairly standard result in the $K$-theory of $C^*$-algebras and can be proved in the same way as it is done for $C^*$-algebras (cf. \cite[Lemma 4.6.2]{HR} or \cite[Section 3, Lemma 2]{HRY}).
\begin{lem} \label{lemMult}
Let $\alpha:A\rightarrow C$ be a bounded homomorphism of $L^p$ operator algebras with $C\subset B(L^p(\mu))$, and let $v\in B(L^p(\mu))$ be an isometric multiplier of $C$. Then the map $a\mapsto v\alpha(a)v^\dagger $ is a bounded homomorphism from $A$ to $C$, and induces the same map as $\alpha$ on $K$-theory.

More generally, if $v\in B(L^p(\mu))$ is a partial isometry and a multiplier of $C$, $w$ is a generalized inverse of $v$ that is also a multiplier of $C$, and $\alpha(a)wv=\alpha(a)=wv\alpha(a)$ for all $a\in A$, then the map $a\mapsto v\alpha(a)w$ is a bounded homomorphism from $A$ to $C$, and induces the same map as $\alpha$ on $K$-theory.
\end{lem}

\begin{lem}
$K_*(M(A^s(G)))=0$.
\end{lem}

\begin{proof}
Note that $\mu,\mu^{+1}:M(A^s(G))\rightarrow M(A^s(G;\mathcal{K}_\Gamma^\infty))$ induce the same map on $K$-theory by the previous lemma. Moreover, since $\mu^0(a)\mu^{+1}(a)=\mu^{+1}(a)\mu^0(a)=0$ for all $a\in M(A^s(G))$ and $\mu=\mu^0+\mu^{+1}$, the induced maps on $K$-theory satisfy $\mu_*=\mu^0_*+\mu^{+1}_*=\mu^0_*+\mu_*$. Hence $\mu^0_*=0$. But $\mu^0$ induces an isomorphism on $K$-theory so $K_*(M(A^s(G)))=0$.
\end{proof}


For each $z\in Z_s$ such that $(z,x)\in P_s(G)$ for some $x\in X$ and for each $r\in\mathbb{Q}\cap[0,1]$, let $E_{z,r}$ be a copy of $\ell^p$ so that we have an isometric isomorphism $\ell^p\cong(\bigoplus_{z,r}E_{z,r})_{p}$, and let $w_{z,r}:\ell^p\otimes\ell^p(\Gamma)\rightarrow \ell^p\otimes\ell^p(\Gamma)$ be an isometry with range $E_{z,r}\otimes\ell^p(\Gamma)$, and let $w_{z,r}^\dagger:\ell^p\otimes\ell^p(\Gamma)\rightarrow\ell^p\otimes\ell^p(\Gamma)$ be the projection onto $E_{z,r}\otimes\ell^p(\Gamma)$. 
For each $r\in\mathbb{Q}\cap[0,1]$, define an isometry \[w(r):\ell^p(Z_G)\otimes\ell^p\otimes\ell^p(\Gamma)\rightarrow\ell^p(Z_G)\otimes\ell^p\otimes\ell^p(\Gamma)\] by $\delta_{z,x}\otimes\eta\mapsto\delta_{(1-r)z+re,x}\otimes w_{z,r}\eta$.

For $r\in\mathbb{Q}\cap[0,1)$, let $w(r)^\dagger$ be given by
\[
\delta_{z,x}\otimes\eta\mapsto\begin{cases} \delta_{z',x}\otimes w_{z',r}^\dagger\eta &\text{if}\; z=(1-r)z'+re, \\ 0 &\text{otherwise}. \end{cases}
\]
and let $w(1)^\dagger$ be given by 
\[
\delta_{z,x}\otimes\eta\mapsto\begin{cases} \sum_{z'}\delta_{z',x}\otimes w_{z',1}^\dagger\eta &\text{if}\;z=e, \\ 0 &\text{otherwise}. \end{cases}
\]
Note that if $r\neq s$, then $w(r)^\dagger w(s)=0$.

For $t\geq 0$ and $n\in\mathbb{N}\cup\{\infty\}$, define an isometry \[v_n(t):\ell^p(Z_G)\otimes\ell^p\otimes\ell^p(\Gamma)\rightarrow\ell^p(Z_G)\otimes\ell^p\otimes\ell^p(\Gamma)\] in the following way.
Set $v_n(t)$ to be $w(0)$ if $t\leq n$, and $w(1)$ if $t\geq 2n$. For $t=m+s$ with $m\in\mathbb{N}\cap(n,2n)$ and $s\in(0,1)$, set
\[ v_n(t)=\cos\left(\frac{s\pi}{2}\right)^{\frac{2}{p}}w\left(\frac{m-n}{n}\right) + \sin\left(\frac{s\pi}{2}\right)^{\frac{2}{p}}w\left(\frac{m+1-n}{n}\right). \]
Note that if $\frac{1}{p}+\frac{1}{q}=1$, then
\[ v_n(t)^\dagger=\cos\left(\frac{s\pi}{2}\right)^{\frac{2}{q}}w\left(\frac{m-n}{n}\right)^\dagger + \sin\left(\frac{s\pi}{2}\right)^{\frac{2}{q}}w\left(\frac{m+1-n}{n}\right)^\dagger \]
satisfies $v_n(t)^\dagger v_n(t)=I$ for $t=m+s$ with $m\in\mathbb{N}\cap(n,2n)$ and $s\in(0,1)$.

One can check that the map $[0,\infty)\rightarrow B(\ell^p(Z_G)\otimes\ell^p\otimes\ell^p(\Gamma)),t\mapsto v_n(t)$, is norm continuous for each $n$. 
Now define an isometry \[v_n:L^p([0,\infty),E_G)\rightarrow L^p([0,\infty),E_G)\] for each $n$ by $(v_n\xi)(t)=v_n(t)(\xi(t))$ for $\xi\in L^p([0,\infty),E_G)$. Also set $(v_n^\dagger\xi)(t)=v_n(t)^\dagger(\xi(t))$ so that $v_n^\dagger v_n=I$.

Let $a\in\mathbb{C}_{L,0}[G;s]$ and let $T=a(t)$ for some fixed $t\in[0,\infty)$. The matrix entries $(v_n(t)Tv_n(t)^\dagger )_{y,z}(x)$ of $v_n(t)Tv_n(t)^\dagger $ will be linear combinations of at most four terms of the form $w_{y,r_1} T_{F_{r_1}(y),F_{r_2}(z)}(x)w_{z,r_2}^\dagger $, with $F_r$ as in (\ref{homodef}) on page \pageref{homodef}, $r_1,r_2\in\mathbb{Q}\cap[0,1]$, and $|r_1-r_2|<\frac{1}{m}$ whenever $t>2m$. It follows that the Rips-propagation of $v_n(t)a(t)v_n(t)^\dagger $ is at most $\mathrm{prop}_{Rips}a(t)+\min(1,\frac{2}{|t-1|})$, so $v_nav_n^\dagger \in\mathbb{C}_{L,0}[G;s]$.

Also, the operators $S_t:=v_{n+1}(t)v_n(t)^\dagger $ on $\ell^p(Z_G)\otimes\ell^p\otimes\ell^p(\Gamma)$ have matrix entries $(S_t)_{y,z}$ that act as constant functions $X\rightarrow B(\ell^p\otimes\ell^p(\Gamma))$, their Rips-propagation tends to zero as $t\rightarrow\infty$, and they have $\Gamma$-propagation at most $s$ for all $t$. Hence $v_{n+1}v_n^\dagger $ is a multiplier of $A^s(G)$ for all $n$.

\begin{lem}
Let $A$ be a unital Banach algebra, and let $I$ be an ideal in $A$. Define the double of $A$ along $I$ to be $D=\{(a,b)\in A\oplus A:a-b\in I\}$. Assume that $A$ has trivial $K$-theory. Then the inclusion $\iota:I\rightarrow D$ given by $a\mapsto(a,0)$ induces an isomorphism in $K$-theory, and the diagonal inclusion $\delta:I\rightarrow D$ given by $a\mapsto(a,a)$ induces the zero map on $K$-theory.
\end{lem}

\begin{proof}
Note that $\iota(I)$ is an ideal in $D$, and $D/\iota(I)$ is isomorphic to $A$ via the second coordinate projection. Since $K_*(A)=0$, it follows from the six-term exact sequence that $\iota$ induces an isomorphism in $K$-theory. 

On the other hand, $\delta$ factors through the diagonal inclusion $A\rightarrow D,a\mapsto(a,a)$, so $\delta$ induces the zero map on $K$-theory since $K_*(A)=0$.
\end{proof}

We shall apply the lemma in the case where $A=M(A^s(G))$ and $I=A^s(G)$ to prove the next proposition.

\begin{prop}
Let $v_n:L^p([0,\infty),E_G)\rightarrow L^p([0,\infty),E_G)$ be as defined above. Then the maps $a\mapsto v_0 av_0^\dagger $ and $a\mapsto v_\infty av_\infty^\dagger $ induce the same map $K_*(A^s(G))\rightarrow K_*(A^s(G))$.
\end{prop}

\begin{proof}
Given $a\in A^s(G)$, define \[\alpha(a)=(\bigoplus_{n=0}^\infty v_nav_n^\dagger ,\bigoplus_{n=0}^\infty v_\infty av_\infty^\dagger )\in A^s(G;\mathcal{K}_\Gamma^\infty)\oplus A^s(G;\mathcal{K}_\Gamma^\infty).\] Also define \[\beta(a)=(\bigoplus_{n=0}^\infty v_{n+1}av_{n+1}^\dagger ,\bigoplus_{n=0}^\infty v_\infty av_\infty^\dagger )\in A^s(G;\mathcal{K}_\Gamma^\infty)\oplus A^s(G;\mathcal{K}_\Gamma^\infty).\] Let $D$ be the double of $M(A^s(G;\mathcal{K}_\Gamma^\infty))$ along $A^s(G;\mathcal{K}_\Gamma^\infty)$, and let \[C=\{(c,d)\in D:d=\bigoplus_{n=0}^\infty v_\infty av_\infty^\dagger \;\text{for some}\;a\in A^s(G)\},\] which is a closed subalgebra of $D$. Moreover, $\alpha$ and $\beta$ are bounded homomorphisms with image in $C$. 
Consider $w=(w_1,w_2)$, where $w_1=\bigoplus_{n=0}^\infty v_{n+1}v_n^\dagger $ and $w_2=\bigoplus_{n=0}^\infty v_\infty v_\infty^\dagger $. Note that $w_1,w_2\in M(A^s(G;\mathcal{K}_\Gamma^\infty))$. We claim that $w$ is a multiplier of $C$. Indeed, if $(c,d)\in C$, then $w_2d=dw_2=d$ so it suffices to show that $cw_1-d$ and $w_1c-d$ are in $A^s(G;\mathcal{K}_\Gamma^\infty)$. We will only consider $w_1c-d$ since the other case is similar. Now $w_1c-d=w_1(c-d)+(w_1d-d)$ so it suffices to show that $w_1d-d\in A^s(G;\mathcal{K}_\Gamma^\infty)$. But $w_1d-d=(w_1-w_2)d=\bigoplus_{n=0}^\infty(v_{n+1}v_n^\dagger -v_\infty v_\infty^\dagger )v_\infty a v_\infty^\dagger \in A^s(G;\mathcal{K}_\Gamma^\infty)$ since $v_n(t)=v_\infty(t)$ for each fixed $t$ and all $n>t$. Similarly, $w^\dagger =(w_1^\dagger ,w_2^\dagger )$ is a multiplier of $C$. 

Now $\beta(a)=w\alpha(a)w^\dagger $ for all $a\in A$. Moreover, $\alpha(a)w^\dagger w=\alpha(a)=w^\dagger w\alpha(a)$ for all $a\in A$ so $\alpha$ and $\beta$ induce the same map $K_*(A^s(G))\rightarrow K_*(C)$ by Lemma \ref{lemMult}, and thus the same map $K_*(A^s(G))\rightarrow K_*(D)$ upon composing with the map induced by the inclusion of $C$ into $D$.

Let $u$ be the isometric multiplier of $A^s(G;\mathcal{K}_\Gamma^\infty)$ induced by the right shift. Then $(u,u)$ is a multiplier of $D$, and conjugating $\beta(a)$ by $(u,u)$ gives \[\gamma(a)=\biggl(0\oplus\bigoplus_{n=1}^\infty v_nav_n^\dagger ,0\oplus\bigoplus_{n=1}^\infty v_\infty av_\infty^\dagger \biggr).\] Thus $\beta$ and $\gamma$ induce the same map $K_*(A)\rightarrow K_*(D)$. On the other hand, the homomorphism $\delta:A^s(G)\rightarrow D$ given by \[a\mapsto(v_\infty av_\infty^\dagger \oplus 0\oplus 0\oplus\cdots,v_\infty av_\infty^\dagger \oplus 0\oplus 0\oplus\cdots)\] induces the zero map on $K$-theory by the previous lemma. Also, $\gamma(a)\delta(a)=\delta(a)\gamma(a)=0$. Hence \[\alpha_*=\beta_*=\gamma_*=\gamma_*+\delta_*=(\gamma+\delta)_*:K_*(A^s(G))\rightarrow K_*(D).\]

Let $\psi_0,\psi_\infty:A^s(G)\rightarrow D$ be the homomorphisms defined by \begin{align*} \psi_0(a) &=(v_0 av_0^\dagger \oplus 0\oplus 0\oplus\cdots,0), \\ \psi_\infty(a) &=(v_\infty av_\infty^\dagger \oplus 0\oplus 0\oplus\cdots,0).\end{align*} Also define $\zeta:A^s(G)\rightarrow D$ by \[\zeta(a)=\biggl(0\oplus\bigoplus_{n=1}^\infty v_nav_n^\dagger ,\bigoplus_{n=0}^\infty v_\infty av_\infty^\dagger \biggr).\] Note that $\zeta(a)\psi_0(a)=\psi_0(a)\zeta(a)=\zeta(a)\psi_\infty(a)=\psi_\infty(a)\zeta(a)=0$ for all $a\in A^s(G)$. Also, $\psi_0+\zeta=\alpha$ and $\psi_\infty+\zeta=\gamma+\delta$. Hence \[(\psi_0)_*+\zeta_*=\alpha_*=(\gamma+\delta)_*=(\psi_\infty)_*+\zeta_*,\] so $\psi_0$ and $\psi_\infty$ induce the same maps on $K$-theory.

Finally, if $\iota:A^s(G)\rightarrow D$ is the inclusion into the first factor (where $D$ is now regarded as the double of $M(A^s(G))$ along $A^s(G)$), then $\psi_i(a)$ is given by the composition
\[ a\mapsto v_iav_i^\dagger \stackrel{\iota}{\mapsto}(v_iav_i^\dagger ,0)\mapsto (v_iav_i^\dagger \oplus 0\oplus 0\oplus\cdots,0), \] and the last two maps induce isomorphisms on $K$-theory, so $a\mapsto v_0av_0^\dagger $ and $a\mapsto v_\infty av_\infty^\dagger$ induce the same map on $K$-theory.
\end{proof}

Now it remains to be shown that $a\mapsto v_\infty av_\infty^\dagger $ induces the identity map on $K_*(A^s(G))$ while $a\mapsto v_0 av_0^\dagger $ induces the zero map on $K_*(A^s(G))$. This will complete the proof of Proposition \ref{baseprop}.

\begin{lem}
The map $\phi_\infty:K_*(A^s(G))\rightarrow K_*(A^s(G))$ induced by conjugation by $v_\infty$ is the identity map.
\end{lem}

\begin{proof}
Since $v_\infty(t)=w(0)$ for all $t$, and $w(0)$ is an isometric multiplier of $A^s(G)$, $\phi_\infty$ induces the identity map on $K$-theory by Lemma \ref{lemMult}.
\end{proof}

\begin{lem}
The map $\phi_0:K_*(A^s(G))\rightarrow K_*(A^s(G))$ induced by conjugation by $v_0$ is the zero map.
\end{lem}

\begin{proof}
Let $G^{(0)}$ be the unit space of $G$, which is an open subgroupoid of $\Gamma\ltimes X$. We may then consider $A^s(G^{(0)})$. In fact, $\phi_0$ factors through $K_*(A^s(G^{(0)}))$, i.e., we have a commutative diagram
\[
\begindc{\commdiag}[100]		
\obj(10,5)[2a]{$K_*(A^s(G))$}	\obj(20,5)[2b]{$K_*(A^s(G))$}	

\mor{2a}{2b}{$\phi_0$}	 

\obj(15,0)[3a]{$K_*(A^s(G^{(0)}))$}	

\mor{2a}{3a}{}	\mor{3a}{2b}{} 	
\enddc
\]
so it suffices to show that $K_*(A^s(G^{(0)}))=0$.

Note that $a(t)$ has zero Rips-propagation for all $a\in A^s(G^{(0)})$ and $t\in[0,\infty)$. For each $n\in\mathbb{N}$ and $a\in A^s(G^{(0)})$, define \[a^{(n)}(t)=\begin{cases} a(t-n) & \text{for}\; t\geq n \\ 0 & \text{for}\; t<n \end{cases}.\]
Note that $a^{(n)}\in A^s(G^{(0)})$. Now define $\alpha:A^s(G^{(0)})\rightarrow A^s(G^{(0)};\mathcal{K}_\Gamma^\infty)$ by $a\mapsto \bigoplus_{n=0}^\infty a^{(n)}$. 
We also have the ``top corner inclusion'' $\iota:A^s(G^{(0)})\rightarrow A^s(G^{(0)};\mathcal{K}_\Gamma^\infty)$ given by $a\mapsto a\oplus 0\oplus 0\oplus\cdots$. 
Using uniform continuity of elements in $A^s(G^{(0)})$, we see that $\alpha_*+\iota_*=\alpha_*$ so $\iota_*=0$. But $\iota$ induces an isomorphism on $K$-theory so it follows that $K_*(A^s(G^{(0)}))=0$.
\end{proof}

\subsection{Mayer-Vietoris} \label{sec:Ind}

Given two open subgroupoids of $\Gamma\ltimes X$, we will consider associated subalgebras of $A^s:=B^p_{L,0}(\Gamma\curvearrowright X;s)$, following \cite[Section 7]{GWY2}, and also controlled Mayer-Vietoris pairs of ideals for these algebras. Note that the filtrations we equip these subalgebras with are not the induced filtrations from $A^s$. 

\begin{defn} \label{idealdef}
Fix an open subgroupoid $G$ of $\Gamma\ltimes X$, and fix $s_0\geq 1$. Let $G_0$ and $G_1$ be open subgroupoids of $G$ such that $G^{(0)}=G_0^{(0)}\cup G_1^{(0)}$. For $r\geq 0$, define
\[ \mathfrak{A}_r(G)=A^{s_0}(G_0^{+r})_r+A^{s_0}(G_1^{+r})_r+A^{s_0}(G_0^{+r}\cap G_1^{+r})_r, \] 
where all expansions are taken relative to $G$,
and define \[ \mathfrak{A}(G)=\overline{\bigcup_{r\geq 0}\mathfrak{A}_r(G)}, \]
taking closure in the norm of $A^{s_0}$.

Also define 
\begin{align*} 
I_r&=A^{s_0}(G_0^{+r})_r+A^{s_0}(G_0^{+r}\cap G_1^{+r})_r, \quad I=\overline{\bigcup_{r\geq 0}I_r} \\ J_r&=A^{s_0}(G_1^{+r})_r+A^{s_0}(G_0^{+r}\cap G_1^{+r})_r, \quad J=\overline{\bigcup_{r\geq 0}J_r}.
\end{align*}

Let $G,G_0,G_1$ and $s_0$ be as above. Define another filtration as follows. For $s\geq s_0$ and $r\geq 0$, define
\[ \mathfrak{A}_r^s(G)=A^{s_0}(G_0^{+r})_r+A^{s_0}(G_1^{+r})_r+A^s(G_0^{+r}\cap G_1^{+r})_{sr}, \]
and \[ \mathfrak{A}^s(G)=\overline{\bigcup_{r\geq 0}\mathfrak{A}_r^s(G)}, \]
taking closure in the norm of $A^s$.

Also define 
\begin{align*} 
I_r^s&=A^{s_0}(G_0^{+r})_r+A^s(G_0^{+r}\cap G_1^{+r})_{sr}, \quad I^s=\overline{\bigcup_{r\geq 0}I_r^s} \\ J_r^s&=A^{s_0}(G_1^{+r})_r+A^s(G_0^{+r}\cap G_1^{+r})_{sr}, \quad J^s=\overline{\bigcup_{r\geq 0}J_r^s}.
\end{align*}
\end{defn}

\begin{lem}
With notation as above, $(\mathfrak{A}_r(G))_{r\geq 0}$ and $(\mathfrak{A}_r^s(G))_{r\geq 0}$ are filtrations for $\mathfrak{A}(G)$ and $\mathfrak{A}^s(G)$ respectively. Moreover, $I$ and $J$ are ideals in $\mathfrak{A}(G)$, while $I^s$ and $J^s$ are ideals in $\mathfrak{A}^s(G)$.
\end{lem}

\begin{proof}
It is clear that $\mathfrak{A}_{r_0}(G)\subset \mathfrak{A}_r(G)$ if $r_0\leq r$, that $\bigcup_{r\geq 0}\mathfrak{A}_r(G)$ is dense in $\mathfrak{A}(G)$, and similarly for the $s$-version. By Lemmas \ref{supplem} and \ref{pdtlem}, it follows that for $r_1,r_2\geq 0$ and $s_0\geq 1$, \[ A^{s_0}(G_i^{+r_1})_{r_1}\cdot A^{s_0}(G_i^{+r_2})_{r_2}\subset A^{s_0}(G_i^{+(r_1+r_2)})_{r_1+r_2} \] for $i=0,1$, while 
\[
A^{s_0}(G_0^{+r_1})_{r_1}\cdot A^{s_0}(G_1^{+r_2})_{r_2} \subset A^{s_0}(G_0^{+(r_1+r_2)})_{r_1+r_2}\cap A^{s_0}(G_1^{+(r_1+r_2)})_{r_1+r_2}, \]\[
A^{s_0}(G_0^{+r_1}\cap G_1^{+r_1})_{r_1}\cdot A^{s_0}(G_0^{+r_2}\cap G_1^{+r_2})_{r_2} \subset A^{s_0}(G_0^{+(r_1+r_2)}\cap G_1^{+(r_1+r_2)})_{r_1+r_2}.
\]
Also, for $s\geq s_0$,
\begin{align*} 
A^s(G_0^{+r_1}\cap G_1^{+r_1})_{sr_1}\cdot A^s(G_0^{+r_2}\cap G_1^{+r_2})_{sr_2} &\subset A^s(G_0^{+(r_1+r_2)}\cap G_1^{+(r_1+r_2)})_{s(r_1+r_2)}
\end{align*} 
and
\begin{align*} A^{s_0}(G_i^{+r_1})_{r_1}\cdot A^s(G_0^{+r_2}\cap G_1^{+r_2})_{sr_2}&\subset A^s((G_0^{+r_2}\cap G_1^{+r_2})^{+r_1})_{r_1+sr_2} \\ &\subset A^s((G_0^{+r_2})^{+r_1}\cap(G_1^{+r_2})^{+r_1})_{s(r_1+r_2)} \\ &\subset A^s(G_0^{+(r_1+r_2)}\cap G_1^{+(r_1+r_2)})_{s(r_1+r_2)}. \end{align*}
Hence $(\mathfrak{A}_r(G))_{r\geq 0}$ and $(\mathfrak{A}_r^s(G))_{r\geq 0}$ are filtrations for $\mathfrak{A}(G)$ and $\mathfrak{A}^s(G)$ respectively, $I$ and $J$ are ideals in $\mathfrak{A}(G)$, while $I^s$ and $J^s$ are ideals in $\mathfrak{A}^s(G)$.
\end{proof}

Now we need to check that the ideals in Definition \ref{idealdef} satisfy the conditions listed in Definition \ref{MVpair} for our controlled Mayer-Vietoris sequence. To do so, we shall make use of partitions of unity and associated multiplication operators.

\begin{defn}
Let $K$ be a compact subset of $X$, let $\{U_0,\ldots,U_d\}$ be a finite open cover of $K$, and let $\{\phi_0,\ldots,\phi_d\}$ be a subordinate partition of unity. Let $s\geq 0$. For $i\in\{0,\ldots,d\}$, let $M_i$ be the multiplication operator on $E_s$ associated to the function \[ Z_s\times X\rightarrow [0,1], (z,x)\mapsto\sum_{g\in\Gamma}t_g(z)\phi_i(gx). \]
\end{defn}

\begin{lem}
With notation as above, the operators $M_i$ have the following properties:
\begin{enumerate}
\item $||M_i||\leq 1$.
\item If $T\in B^p(\Gamma\curvearrowright X;s)$ satisfies \[\{x\in X:(gx,g,x)\in \supp(T)\;\text{for some}\;g\in\Gamma\}\subset K,\] then $T=T(M_0+\cdots+M_d)$.
\item For any $T\in B^p(\Gamma\curvearrowright X;s)$ with $\Gamma$-propagation at most $r$, and $i\in\{0,\ldots,d\}$, we have \[ \supp(TM_i)\subset\biggl\{ (gx,g,x)\in\Gamma\ltimes X:x\in\bigcup_{l(h)\leq s}h\cdot U_i,l(g)\leq r\biggr\}\cap \supp(T). \]
\end{enumerate}
\end{lem}

\begin{proof}
Each $M_i$ is a multiplication operator associated to a function taking values in $[0,1]$ so it follows that $||M_i||\leq 1$.

For $i\in\{0,\ldots,d\}$, $T\in B^p(\Gamma\curvearrowright X;s)$, $y,z\in P_s(\Gamma)$, and $x\in X$, we have \[ (TM_i)_{y,z}(x)=T_{y,z}(x)\cdot\sum_{h\in\Gamma}t_h(z)\phi_i(hx). \] 
Hence \[ (T(M_0+\cdots+M_d))_{y,z}(x)=T_{y,z}(x)\cdot\sum_{h\in\Gamma}t_h(z)(\phi_0(hx)+\cdots+\phi_d(hx)). \]

Suppose that $\{x\in X:(gx,g,x)\in \supp(T)\;\text{for some}\;g\in\Gamma\}\subset K$. If $T_{y,z}(x)\neq 0$, then $(gx,gh^{-1},hx)\in \supp(T)$ for all $g\in \supp(y)$ and $h\in \supp(z)$. 
In particular, $hx\in K$ for all $h\in \supp(z)$, so \[ \sum_{h\in\Gamma}t_h(z)(\phi_0(hx)+\cdots+\phi_d(hx))=\sum_{h\in\Gamma}t_h(z)=1, \]
and this proves (ii).

Suppose that $(gx,gk^{-1},kx)\in \supp(TM_i)$, where $T\in B^p(\Gamma\curvearrowright X;s)$ has $\Gamma$-propagation at most $r$. Then there exist $y,z\in P_s(\Gamma)$ with $g\in \supp(y)$, $k\in \supp(z)$, and $(TM_i)_{y,z}(x)\neq 0$. In particular, $T_{y,z}(x)\neq 0$, so $(gx,gk^{-1},kx)\in \supp(T)$ and $l(gk^{-1})\leq r$. We also have $\sum_{h\in\Gamma}t_h(z)\phi_i(hx)\neq 0$, so there exists $h\in \supp(z)$ with $\phi_i(hx)\neq 0$, and thus $hx\in U_i$. Since $h,k\in \supp(z)$, and $z\in P_s(\Gamma)$, we have $l(kh^{-1})\leq s$. Now $kx=(kh^{-1})hx\in kh^{-1}\cdot U_i$, and this proves (iii).
\end{proof}

\begin{lem}
The pairs $(I,J)$ and $(I^s,J^s)$ in Definition \ref{idealdef} are controlled Mayer-Vietoris pairs for $\mathfrak{A}(G)$ and $\mathfrak{A}^s(G)$ respectively.
\end{lem}

\begin{proof}
By virtue of how $I,J,I^s,J^s$, and $\mathfrak{A}(G),\mathfrak{A}^s(G)$ are defined, we only need to check conditions (i) and (iii) in Definition \ref{MVpair}, which we will do for $I,J,\mathfrak{A}(G)$ only since the $s$-version can be handled in a similar manner.

Let $U_i$ be the unit space of $G_i^{+r_0}$ for $i=0,1$. If $a\in \mathfrak{A}_{r_0}(G)$, then \[ K:=\overline{\{x\in X:(gx,g,x)\in \supp(a(t))\;\text{for some}\;t\in[0,\infty),g\in\Gamma\}} \] is a compact subset of $U_0\cup U_1$. Let $M_0,M_1$ be the multiplication operators defined with respect to the compact set $K$, the open cover $\{U_0,U_1\}$, and some choice of subordinate partition of unity $\{\phi_0,\phi_1\}$. By the previous lemma, we have $a(t)(M_0+M_1)=a(t)$ for all $t$. Moreover, $||a(t)M_i||\leq||a(t)||$ for $i=0,1$. It remains to be shown that $t\mapsto a(t)M_0$ is in $I_r$ and $t\mapsto a(t)M_1$ is in $J_r$ for some $r\geq r_0$ (that may depend on $s_0$ but not on $s$). We will focus on the case of $M_0$ since the other case is similar.

Write $a=b_0+b_1+c$ with $b_i\in A^{s_0}(G_i^{+r_0})_{r_0}$ and $c\in A^s(G_0^{+r_0}\cap G_1^{+r_0})$. By the previous lemma, we have $\supp(b_0(t)M_0) \subset \supp(b_0(t))$ and also $\supp(c(t)M_0)\subset \supp(c(t))$ so $t\mapsto b_0(t)M_0$ is in $A^{s_0}(G_0^{+r_0})_{r_0}\subset I_{r_0}$ and $t\mapsto c(t)M_0$ is in $A^s(G_0^{+r_0}\cap G_1^{+r_0})\subset I_{r_0}$. 

Now assume that $(gx,gh^{-1},hx)\in \supp(b_1(t)M_0)$ for some $t$. Then there exist $y,z\in P_s(\Gamma)$ such that $g\in \supp(y),h\in \supp(z)$, and $(b_1(t)M_0)_{y,z}(x)\neq 0$. In particular, $(b_1(t))_{y,z}(x)\neq 0$ so $y,z\in P_{s_0}(\Gamma)$ and $l(gh^{-1})\leq r_0$. Also, $\sum_{k\in\Gamma}t_k(z)\phi_0(kx)\neq 0$ so there exists $k\in \supp(z)$ such that $\phi_0(kx)\neq 0$, and thus $kx$ is in $U_0$, the unit space of $G_0^{+r_0}$. Hence \begin{align*} (gx,gh^{-1},hx)&=(gx,gk^{-1},kx)(kx,kh^{-1},hx) \\ &\in (G_0^{+r_0})^{+r_0}\cdot (G_0^{+r_0})^{+s_0} \subset G_0^{+(2r_0+s_0)}. \end{align*}
Hence $t\mapsto b_1(t)M_0$ is in $A^{s_0}(G_0^{+(2r_0+s_0)})_{2r_0+s_0}\subset I_{2r_0+s_0}$, and so $t\mapsto a(t)M_0$ is in $I_{2r_0+s_0}$. Hence condition (i) is satisfied.

Next, suppose that $a_0\in I_{r_0}$ and $a_1\in J_{r_0}$ such that $||a_0-a_1||<\varepsilon$. Again, let $U_i$ be the unit space of $G_i^{+r_0}$ for $i=0,1$. Consider
\[ K_i:=\overline{\{x\in X:(gx,g,x)\in \supp(a_i(t))\;\text{for some}\;t\in[0,\infty),g\in\Gamma\}} \]
for $i=0,1$, and let $K=K_1\cup K_2$, a compact subset of $U_0\cup U_1$. Let $M_0,M_1$ be the multiplication operators defined with respect to the compact set $K$, the open cover $\{U_0,U_1\}$, and some choice of subordinate partition of unity $\{\phi_0,\phi_1\}$. Define $b(t)=a_0(t)M_1+a_1(t)M_0$. Then $b\in I_{2r_0+s_0}\cap J_{2r_0+s_0}$, 
and since $a_i(t)=a_i(t)(M_0+M_1)$ by the choice of $K$, we have $||a_i(t)-b(t)||\leq||a_0(t)-a_1(t)||<\varepsilon$ for $i=0,1$. Hence condition (iii) is satisfied.
\end{proof}

The following lemma follows from a similar argument as the first part of the proof above, and we omit the details.

\begin{lem} \label{split}
Fix an open subgroupoid $G$ of $\Gamma\ltimes X$ and $s_0\geq 0$. Let $r_0\geq 0$. Let $G^{(0)}=U_0\cup U_1$ be an open cover of $G^{(0)}$, and let $G_i$ be the subgroupoid of $G$ generated by \[ \{(gx,g,x)\in G:x\in U_i,l(g)\leq r_0\} \] for $i=0,1$. Then $A^{s_0}(G)_{r_0}\subseteq A^{s_0}(G_0^{+(2r_0+s_0)})_{2r_0+s_0}+A^{s_0}(G_1^{+(2r_0+s_0)})_{2r_0+s_0}$.
\end{lem}

Now we can prove our main theorem.

\begin{thm} \label{mainthm}
Suppose that $\Gamma\curvearrowright X$ has finite dynamical complexity, and let $p\in[1,\infty)$. Then \[ \lim_{s\rightarrow\infty}K_*(B^p_{L,0}(\Gamma\curvearrowright X;s))=0. \]
Thus the $L^p$ assembly map for $\Gamma\curvearrowright X$ in Definition \ref{assembly} is an isomorphism. 
\end{thm}

\begin{proof}
As above, we use the shorthand $A^s(G)$ for $B^p_{L,0}(G;s)$. We need to show that for any $s_0\geq 0$ and any $x\in K_*(A^{s_0}(\Gamma\ltimes X))$, there is $s\geq s_0$ such that the map $K_*(A^{s_0}(\Gamma\ltimes X))\rightarrow K_*(A^s(\Gamma\ltimes X))$ induced by inclusion sends $x$ to 0.

Let $\mathcal{C}$ be the class of open subgroupoids $G$ of $\Gamma\ltimes X$ such that for any open subgroupoid $H$ of $G$, for any $s_0\geq 0$, and any $x\in K_*(A^{s_0}(H))$, there exists $s\geq s_0$ such that the map $K_*(A^{s_0}(H))\rightarrow K_*(A^s(H))$ sends $x$ to 0. It suffices to show that $\mathcal{D}\subseteq\mathcal{C}$, where $\mathcal{D}$ is as in Definition \ref{FDCbasic}. By Lemma \ref{FDCstrong}, it suffices to show that $\mathcal{C}$ contains all relatively compact open subgroupoids of $\Gamma\ltimes X$, and that $\mathcal{C}$ is closed under strong decomposability. Note that if $G$ is in $\mathcal{C}$, then so is any open subgroupoid of $G$.

Let $G$ be a relatively compact open subgroupoid of $\Gamma\ltimes X$, let $H$ be an open subgroupoid of $G$, let $s_0\geq 0$, and let $x\in K_*(A^{s_0}(H))$. Set \[N_0=\begin{cases} ||e||+1 &\text{if}\; x=[e]\in K_0(A^{s_0}(H)) \\ ||u||+||u^{-1}||+1 &\text{if}\; x=[u]\in K_1(A^{s_0}(H)) \end{cases},\] and let $P$ be as in Proposition \ref{qKtoKinj}.
By Proposition \ref{qKtoKsurj}, there exists $r_0\geq 0$ such that $x=c_*(y)$ for some $y\in K_*^{\frac{1}{20P(N_0)},r_0,N_0}(A^{s_0}(H))$. Note that $s_1=\max\{l(g):(gx,g,x)\in H\}<\infty$ since $G$ is relatively compact. Set $s=\max(r_0,s_0,s_1)$. By Proposition \ref{baseprop}, $K_*(A^s(H))=0$. Also, $A^s(H)_s=A^s(H)$. Thus by Proposition \ref{qKtoKinj}, for any $z\in K_*^{\frac{1}{20P(N_0)},s,N_0}(A^s(H))$, there exists $N\geq N_0$ such that the image of $z$ in $K_*^{\frac{1}{20},s,N}(A^s(H))$ is zero. In particular, the image of $y$ under the composition
\[ K_*^{\frac{1}{20P(N_0)},r_0,N_0}(A^{s_0}(H))\rightarrow K_*^{\frac{1}{20P(N_0)},s,N_0}(A^s(H))\rightarrow K_*^{\frac{1}{20},s,N}(A^s(H)) \] is zero, so the image of $x=c_*(y)$ in $K_*(A^s(H))$ is zero. Hence, all relatively compact open subgroupoids of $\Gamma\ltimes X$ are in the class $\mathcal{C}$.

Next, we show that the class $\mathcal{C}$ is closed under strong decomposability. 
Let $(\lambda,h)\leq(\lambda',h')$ be the control pairs in Theorem \ref{MVthm}. Suppose that $G$ is an open subgroupoid of $\Gamma\ltimes X$ that is strongly decomposable over $\mathcal{C}$, let $H$ be an open subgroupoid of $G$, let $s_0\geq 0$, and let $x\in K_1(A^{s_0}(H))$. Set $N_0=||u||+||u^{-1}||+1$ if $x=[u]$, and fix $\varepsilon_0\in (0,\frac{1}{20\lambda'_{N_0}P(\lambda_{N_0})})$. By Proposition \ref{qKtoKsurj}, there exists $r_0\geq 0$ such that $x=c_1(y)$ for some $y\in K_1^{\varepsilon_0,r_0,N_0}(A^{s_0}(H))$. 
Set $r_1=h_{\varepsilon_0,N_0}(2r_0+s_0)$ and $r_2=h'_{\varepsilon_0,N_0}r_0$. Since $H$ is strongly decomposable over $\mathcal{C}$, and open subgroupoids of members of $\mathcal{C}$ are also in $\mathcal{C}$, there is an open cover $H^{(0)}=U_0\cup U_1$ such that if $H_i$ is the subgroupoid of $H$ generated by $\{ (gx,g,x)\in H:x\in U_i,l(g)\leq r_0\}$, then $H_i^{+r_2}$ and $H_0^{+r_1}\cap H_1^{+r_1}$ are in $\mathcal{C}$. 

By Lemma \ref{split}, we may regard $y$ as an element in $K_1^{\varepsilon_0,2r_0+s_0,N_0}(\mathfrak{A}(H))$, and we have $\partial_c(y)\in K_0^{\lambda_{N_0}\varepsilon_0,r_1,\lambda_{N_0}}(I\cap J)=K_0^{\lambda_{N_0}\varepsilon_0,r_1,\lambda_{N_0}}(A^{s_0}(H_0^{+r_1}\cap H_1^{+r_1}))$. Since $H_0^{+r_1}\cap H_1^{+r_1}$ belongs to $\mathcal{C}$, there exists $s_1\geq s_0$ such that the image of $c_0(\partial_c(y))$ in $K_0(A^{s_1}(H_0^{+r_1}\cap H_1^{+r_1}))$ is zero, so by Proposition \ref{qKtoKinj}, there exist $s_2\geq s_1$ and $N_1\geq\lambda_{N_0}$ such that $\partial_c(y)=0$ in $K_0^{P(\lambda_{N_0})\lambda_{N_0}\varepsilon_0,s_2,N_1}(A^{s_1}(H_0^{+r_1}\cap H_1^{+r_1}))$, and thus also in $K_0^{P(\lambda_{N_0})\lambda_{N_0}\varepsilon_0,s_2r_1,N_1}(A^{s_2}(H_0^{+r_1}\cap H_1^{+r_1}))$, which may be identified with $K_0^{P(\lambda_{N_0})\lambda_{N_0}\varepsilon_0,r_1,N_1}(I^{s_2}\cap J^{s_2})$.

Now regarding $y$ as an element in $K_1^{P(\lambda_{N_0})\varepsilon_0,r_0,N_0}(\mathfrak{A}^{s_2}(H))$, by Theorem \ref{MVthm}, there exist $a\in K_1^{\lambda'_{N_0}P(\lambda_{N_0})\varepsilon_0,r_2,\lambda'_{N_0}}(I^{s_2})$ and $b\in K_1^{\lambda'_{N_0}P(\lambda_{N_0})\varepsilon_0,r_2,\lambda'_{N_0}}(J^{s_2})$ such that $y=a+b$ in $K_1^{\lambda'_{N_0}P(\lambda_{N_0})\varepsilon_0,r_2,\lambda'_{N_0}}(\mathfrak{A}^{s_2}(H))$. 

Next, regard $a$ as an element in $K_1^{\lambda'_{N_0}P(\lambda_{N_0})\varepsilon_0,s_1r_2,\lambda'_{N_0}}(A^{s_2}(H_0^{+r_2}))$, and $b$  as an element in $K_1^{\lambda'_{N_0}P(\lambda_{N_0})\varepsilon_0,s_1r_2,\lambda'_{N_0}}(A^{s_2}(H_1^{+r_2}))$. Since $H_0^{+r_2}$ and $H_1^{+r_2}$ belong to $\mathcal{C}$, there exists $s_3\geq s_2$ such that the respective images of $c_1(a)$ and $c_1(b)$ are zero in $K_1(A^{s_3}(H_0^{+r_2}))$ and $K_1(A^{s_3}(H_1^{+r_2}))$ respectively, so there exist $s_4\geq s_1r_2$ and $N_2\geq\lambda'_{N_0}$ such that the respective images of $a$ and $b$ are zero in $K_1^{\lambda'_{N_0}P(\lambda_{N_0})\varepsilon_0,s_4,N_2}(A^{s_3}(H_0^{+r_2}))$ and $K_1^{\lambda'_{N_0}P(\lambda_{N_0})\varepsilon_0,s_4,N_2}(A^{s_3}(H_1^{+r_2}))$. Hence the image of $y$ in $K_1^{\lambda'_{N_0}P(\lambda_{N_0})\varepsilon_0,s_4,N_2}(A^{s_3}(H))$ is zero, and the image of $x=c_1(y)$ in $K_1(A^{s_3}(H))$ is zero, concluding the proof in the odd case.

In the even case, essentially the same argument works by considering suspensions, and we omit the details.
\end{proof}

\section{Remarks and Questions}

In this final section, we make some remarks and pose a question about the domain of our $L^p$ assembly map. We also briefly discuss an alternative $L^p$ assembly map involving $L^p$ Roe $\ast$-algebras.

\subsection{The domain of the $L^p$ assembly map}

In this paper, we have formulated an $L^p$ assembly map based on a particular model of the Baum-Connes assembly map in which the domain involves the $K$-theory of some localization algebras. It can be shown that for each $p$ the $K$-theory of these $L^p$ localization algebras has the same homological properties as their $p=2$ counterpart.
However, the relationship between the domain of the original Baum-Connes assembly map and the domain of our $L^p$ assembly map is not clear.
For instance, while we have a Mayer-Vietoris sequence for the $K$-theory of the $L^p$ localization algebra for each $p$, it does not seem clear how we can connect the sequences for different $p$ in order to use a five lemma argument.
Hence we pose the following question:

\begin{qn}
Can the domain of our $L^p$ assembly map be identified with that of the original Baum-Connes assembly map?
\end{qn}

An affirmative answer to this question will imply that the $K$-theory of $C(X)\rtimes_{\lambda,p}\Gamma$ is independent of $p$ for actions with finite dynamical complexity.

\subsection{An alternative assembly map}

Here, we outline an alternative approach to the study of $L^p$ analogs of Baum-Connes type assembly maps in which the domain can be identified with that of the original Baum-Connes assembly map, but which leads to a question regarding the codomain. In this approach, we consider $L^p$ Roe algebras equipped with an involution, which are defined analogously to the $L^p$ group algebras studied in \cite{LY}.

For simplicity, we will consider a discrete metric space $X$ with bounded geometry. In our case of interest, the relevant metric space is the Rips complex of a countable discrete group, and this metric space has a net with bounded geometry, so we do not lose too much generality. 

We first consider the uniform Roe algebra.
If $T=(T_{xy})_{x,y\in X}\in B(\ell^p(X))$ has propagation $r$, then there exists $c_r\geq 0$ such that \[\sup_{x,y\in X}|T_{xy}|\leq ||T||_{B(\ell^p(X))}\leq c_r\sup_{x,y\in X}|T_{xy}|.\]
Since the entries of the matrix $(T_{xy})$ are uniformly bounded and $X$ has bounded geometry, $(T_{xy})$ also represents a bounded operator on $\ell^q(X)$, where $\frac{1}{p}+\frac{1}{q}=1$.
Define \[||T||=\max(||T||_{B(\ell^p(X))},||T||_{B(\ell^q(X))}),\]
and define $B^{p,*}_u(X)$ to be the completion of the set of all finite propagation $X\times X$ matrices with uniformly bounded entries under the norm $||\cdot||$.
Then $B^{p,*}_u(X)$ is a Banach $\ast$-algebra, and there is a contractive homomorphism \[B^{p,*}_u(X)\rightarrow B^p_u(X).\]
Using complex interpolation, there is also a contractive homomorphism \[B^{p,*}_u(X)\rightarrow B^{2,*}_u(X)=C^\ast_u(X).\]

Now for the Roe algebra, if $T\in B(\ell^p(X,\ell^p))$ has propagation $r$, then
\[\sup_{x,y\in X}||T_{xy}||_{B(\ell^p)}\leq ||T||_{B(\ell^p(X,\ell^p))}\leq c_r\sup_{x,y\in X}||T_{xy}||_{B(\ell^p)}.\]
But since $T_{xy}\in K(\ell^p)$ is, \textit{a priori}, not a bounded operator on $\ell^q$, we can only regard $T=(T_{xy})\in B^p(X)$ as a bounded operator on $\ell^q(X,\ell^p)$ rather than $\ell^q(X,\ell^q)$, so there is a problem defining a norm as above to get a Banach $\ast$-algebra. Thus, we propose the following alternative.

Consider the algebra of all finite propagation $X\times X$ matrices $T=(T_{xy})$ with entries $T_{xy}\in M_\infty(\mathbb{C})$ that are uniformly bounded both in $B(\ell^p)$ and in $B(\ell^q)$, where $\frac{1}{p}+\frac{1}{q}=1$. Then $T_{xy}\in K(\ell^p)\cap K(\ell^q)$ for all $x,y\in X$, and $T\in B(\ell^p(X,\ell^p))\cap B(\ell^q(X,\ell^q))$.
Define \[ ||T||=\max(||T||_{B(\ell^p(X,\ell^p))},||T||_{B(\ell^q(X,\ell^q))}), \]
and define $B^{p,*}(X)$ to be the completion of this algebra under the norm $||\cdot||$.

Then $B^{p,*}(X)$ is a Banach $\ast$-algebra, there is a contractive homomorphism \[B^{p,*}(X)\rightarrow B^p(X)\] for all $p\in(1,\infty)$, and there is also a contractive homomorphism \[B^{p,*}(X)\rightarrow B^{2,*}(X)=C^\ast(X).\]

If there is a proper, cocompact action of a discrete group $\Gamma$ on $X$, one can define equivariant versions $B^{p,\ast}(X)^\Gamma$ of these $\ast$-algebras by considering $\Gamma$-invariant matrices. There are also homomorphisms analogous to those in the non-equivariant case.
In our case of interest, $X$ is the Rips complex $P_d(\Gamma)$ of a discrete group $\Gamma$.

Using these Roe $\ast$-algebras, we get the corresponding localization algebras $B^{p,*}_L(X)^\Gamma$.
The interpolation homomorphisms between the Roe $\ast$-algebras induce interpolation homomorphisms
\[ B^{p,\ast}_L(X)^\Gamma\rightarrow B^{2,\ast}_L(X)^\Gamma=C^\ast_L(X)^\Gamma \]
between the corresponding localization algebras.

We claim that $B^{p,*}_L(P_d(\Gamma))^\Gamma\rightarrow C^\ast_L(P_d(\Gamma))^\Gamma$ induces an isomorphism on $K$-theory. For each of these algebras, there is a Mayer-Vietoris sequence associated with decompositions of the finite-dimensional simplicial complex $P_d(\Gamma)$, so verification of the claim boils down to verifying it for the case of the 0-skeleton of $P_d(\Gamma)$, which is identified with $\Gamma$. 
In other words, we need to show that $B^{p,*}_L(|\Gamma|)^\Gamma\rightarrow C^\ast_L(|\Gamma|)^\Gamma$ induces an isomorphism on $K$-theory, where we write $|\Gamma|$ to indicate that we are regarding $\Gamma$ as a metric space. In this case, note that $C^\ast_L(|\Gamma|)^\Gamma\cong C_{ub}([0,\infty),K(\ell^2))$, and evaluation at zero induces an isomorphism $K_*(C_{ub}([0,\infty),K(\ell^2)))\cong K_*(K(\ell^2))$ since an Eilenberg swindle argument can be used to show that the kernel of the evaluation-at-zero map has trivial $K$-theory. 
Similarly, one can show that evaluation at zero induces an isomorphism $K_*(B^{p,*}_L(|\Gamma|)^\Gamma)\cong K_*(K(\ell^p))$. 
Since $K_*(K(\ell^p))\cong K_*(K(\ell^2))$, we have an isomorphism $K_*(B^{p,*}_L(|\Gamma|)^\Gamma)\stackrel{\cong}{\rightarrow} K_*(C^\ast_L(|\Gamma|)^\Gamma)$.
Hence we have the following commutative diagram relating the $\ast$-algebraic version of the $L^p$ assembly map $\mu_{p,\ast}$ in the top row with the original Baum-Connes assembly map $\mu_{2,\ast}$ in the bottom row:
\[
\begindc{\commdiag}[100]		
\obj(0,5)[1a]{$\lim_{d\rightarrow\infty}K_*(B^{p,*}_L(P_d(\Gamma))^\Gamma)$}		
\obj(15,5)[1b]{$\lim_{d\rightarrow\infty}K_*(B^{p,*}(P_d(\Gamma))^\Gamma)$}
\obj(26,5)[1c]{$K_*(B^{p,*}(|\Gamma|)^\Gamma)$}
\obj(0,0)[2a]{$\lim_{d\rightarrow\infty}K_*(C^\ast_L(P_d(\Gamma))^\Gamma)$}
\obj(15,0)[2b]{$\lim_{d\rightarrow\infty}K_*(C^\ast(P_d(\Gamma))^\Gamma)$}
\obj(26,0)[2c]{$K_*(C^\ast(|\Gamma|)^\Gamma)$}

\mor{1a}{2a}{$\cong$} \mor{1c}{2c}{} \mor{1a}{1b}{} \mor{1b}{1c}{} \mor{2a}{2b}{} \mor{2b}{2c}{}
\enddc
\]
As a consequence, if $\mu_{p,\ast}$ is an isomorphism for all $p\in(1,\infty)$, then \[K_*(B^{p,*}(|\Gamma|)^\Gamma)\cong K_*(C^\ast(|\Gamma|)^\Gamma)\cong K_*(C_r^\ast(\Gamma))\] for all $p\in(1,\infty)$.
If the homomorphism $B^{p,*}(|\Gamma|)^\Gamma\rightarrow B^p(|\Gamma|)^\Gamma$ also induces an isomorphism on $K$-theory for all $p\in(1,\infty)$, then \[K_*(B^p(|\Gamma|)^\Gamma)\cong K_*(C_r^\ast(\Gamma))\] for all $p\in(1,\infty)$. This leads to the following questions:

\begin{qn}
For which groups $\Gamma$ is the assembly map $\mu_{p,\ast}$ an isomorphism?
\end{qn}

\begin{qn}
For which groups $\Gamma$ is it true that the homomorphism \[B^{p,*}(|\Gamma|)^\Gamma\rightarrow B^p(|\Gamma|)^\Gamma\] induces an isomorphism on $K$-theory?
\end{qn}

Note that $B^p(|\Gamma|)^\Gamma$ is isomorphic to $B^p_r(\Gamma)\otimes K(\ell^p)$, where $B^p_r(\Gamma)$ is the reduced $L^p$ group algebra of $\Gamma$ (also known as the algebra of $p$-pseudofunctions on $\Gamma$) \cite[Lemma 6]{CN}. The same proof shows that $B^{p,\ast}(|\Gamma|)^\Gamma$ is isomorphic to to $B^{p,\ast}_r(\Gamma)\otimes K(\ell^p)$, where $B^{p,\ast}_r(\Gamma)$ is an involutive version of $B^p_r(\Gamma)$. These algebras were studied in \cite{LY}. 
In particular, it was shown that for $p\in[1,\infty]$, $\frac{1}{p}+\frac{1}{q}=1$, $q\leq q_0\leq\infty$, if $G$ is a locally compact group with property $(RD)_{q_0}$, then the homomorphism $B^{p,\ast}_r(G)\rightarrow B^p_r(G)$ induces an isomorphism on $K$-theory \cite[Corollary 4.9]{LY}. It was also shown that for a locally compact group $G$ in Lafforgue's class $\mathcal{C}'$ (defined in \cite{Laf02}) having property $(RD)_q$, the $K$-theory of $B^p_r(G)$ does not depend on the parameter $p\in[1,\infty]$ \cite[Corollary 4.10]{LY}.

We shall end with some brief remarks about the case with coefficients.
Defining Roe $\ast$-algebras with general $L^p$ operator algebra coefficients using norms like those above may be problematic as a given $L^p$ operator algebra may not be representable on the dual $L^q$ space, but this is not a problem for algebras of the form $C_0(Y)$ for a locally compact space $Y$.
One may consider the algebra of all finite propagation $X\times X$ matrices $T=(T_{xy})$ with entries $T_{xy}\in C_0(Y)\odot M_\infty(\mathbb{C})$ that are uniformly bounded both in $B(L^p(Y)\otimes\ell^p)$ and in $B(L^q(Y)\otimes\ell^q)$ so that $T\in B(\ell^p(X)\otimes L^p(Y)\otimes\ell^p)\cap B(\ell^q(X)\otimes L^q(Y)\otimes\ell^q)$.
Then define \[ ||T||=\max(||T||_{B(\ell^p(X)\otimes L^p(Y)\otimes\ell^p)},||T||_{B(\ell^q(X)\otimes L^q(Y)\otimes\ell^q)}), \]
and define $B^{p,\ast}(X;C_0(Y))$ to be the completion of this algebra under the norm $||\cdot||$.
One can then consider similar homomorphisms and questions as those discussed above.

\medskip

\emph{Acknowledgements.} This project was started when the author was a Ph.D. student at Texas A\&M University, and the author would like to thank Guoliang Yu for his guidance. The author also thanks Rufus Willett for explaining various aspects of their work in the $C^*$-algebraic setting, N. Christopher Phillips for answering questions about $L^p$ operator algebras, and the referee for comments that helped to improve this paper.

\bibliographystyle{plain}
\bibliography{mybib}
\end{document}